\documentclass[a4paper,12pt]{article}
\usepackage[UKenglish]{babel}

\usepackage[utf8]{inputenc}
\usepackage[margin=1in]{geometry}

\usepackage{amsmath, amssymb, amsthm, mathtools}
\usepackage{tikz} % for small circle
\newcommand{\circo}{~\raisebox{1pt}{\tikz \draw[line width=0.6pt] circle(1.1pt);}~}

\usepackage[toc,page]{appendix} 
\usepackage{chngcntr}

\usepackage{graphicx}
\graphicspath{ {./} }

\usepackage[backend=bibtex,style=authoryear,natbib=true, maxcitenames=1,url=true]{biblatex} % for biblatex
\newtheorem{theorem}{Theorem} % JRSS B require not by section, but individual type count
\newtheorem{lemma}{Lemma}

\newcommand{\mbbR}{\mathbb{R}}
\newcommand{\mbbZ}{\mathbb{Z}}

\newcommand{\mcalA}{\mathcal{A}}
\newcommand{\mcalB}{\mathcal{B}}

\newcommand{\itJ}{\textit{J}}
\newcommand{\bfx}{\textbf{x}}
\newcommand{\bfy}{\textbf{y}}

\newcommand{\bfH}{\textbf{H}}

\newcommand{\bfM}{\textbf{M}}
\newcommand{\bfR}{\textbf{R}}
\newcommand{\bfX}{\textbf{X}}
\newcommand{\bfY}{\textbf{Y}}
\newcommand{\bfB}{\textbf{B}}
\newcommand{\bfZ}{\textbf{Z}}
\newcommand{\bfs}{\textbf{s}}

\newcommand{\mcalG}{\mathcal{G}}
\newcommand{\indep}{\perp \!\!\! \perp}
\newcommand{\nindep}{\not\!\perp\!\!\!\perp}
\newcommand{\mcalS}{\mathcal{S}}
\newcommand{\ttilde}{\Tilde{t}}
\newcommand{\mfkG}{\mathfrak{G}}
\newcommand{\mcalC}{\mathcal{C}}
\newcommand{\mcalN}{\mathcal{N}}
\newcommand{\mcalR}{\mathcal{R}}
\newcommand{\sigalg}{$\sigma$-algebras}

\usepackage[affil-it]{authblk} 
\usepackage{etoolbox}
\usepackage{lmodern}

\usepackage{lipsum}
\patchcmd\abstract{\small}{\footnotesize}{}{}
\providecommand{\keywords}[1]
{
  \small	
  \textbf{\textit{Keywords---}} #1
}

\title{On the Stochasticity of Reanalysis Outputs of 4D-Var}
\author[1,2]{Xiaoqing Chen}
\author[3]{Ross Bannister}
\author[1,2]{Gavin Shaddick}
\author[4]{James V. Zidek}

\affil[1]{The Alan Turing Institute;}
\affil[2]{Department of Mathematics and Statistics, %Faculty of Environment, Science and Economy, 
University of Exeter}
\affil[3]{Department of Meteorology, University of Reading}
\affil[4]{Department of Statistics, University of British Columbia}
\affil[]{\textit{xchen@turing.ac.uk}}
\date{\vspace{-3ex}}

\begin{document}
\maketitle
\sloppy

\begin{abstract}
%\normalsize
    %\lipsum[1]
        % simple, positive, punchy 
    % address: 
            % what question are you asking?
            % within which experiment system are you working?
            % what are your results?
            % what's your answer to the question posed?
    % field of study: start broadly by introducing your field of study
    % aim: what have you been trying to show, why
    % your method: how did you do
    % results: achieved
    % impact: importance of your work in relation to the rest of your field of study
% field of study: ECMWF data why it's important?
This work is motivated by the ECMWF CAMS reanalysis data, a valuable resource for researchers in environmental-related areas, as they contain the most updated atmospheric composition information on a global scale.  
%have become an important source of atmospheric environmental information for researchers and policy makers to address environmental concerns as well as their relevant impacts. 
Unlike observational data obtained from monitoring equipment, such reanalysis data are produced by computers via a 4D-Var data assimilation mechanism, thus their stochastic property remains largely unclear.
% induce of problem: However what remains unclear , hinders wider and flexible aplications in terms of spati0-temporal statistical modelling  
Such lack of knowledge in turn limits their utility scope and 
hinders them from wider and more flexible statistical usages, especially spatio-temporal modelling except for uncertainty quantification and data fusion. 
% aim: To that end
Therefore, this paper studies the stochastic property of these reanalysis outputs data.  
% method: we have done what, at the heart of which is 
We used measure theory and proved tangible existence of spatial and temporal stochasticity associated with these reanalysis data 
%At the heart of the proof is the
%stochastic process constructing theory which has two constructing paths. 
% results: 
and revealed that they are essentially realisations from digitised versions of real-world hidden spatial and/or temporal stochastic processes.  
% impact: 
This means we can treat the reanalysis outputs data the same as observational data in practice and thus ensures more flexible spatio-temporal stochastic methodologies apply to them.
We also objectively analysed different types of errors in the reanalysis data and deciphered their mutual dependence/independence, which together give clear and definite guidance on the modelling of error terms. 
%Altogether, we expand the utility scope of this reanalysis data set beyond empirical utilities such as climatological computing (e.g., mean, percentile, etc.), trends studying, geographical visualisation mapping, etc.  and limited statistical applications
%such as uncertainty quantification and data fusion. 
%More spatio-temporal stochastic models and methodologies are applicable to them through which the reanalysis data set can be better further, and the refined result may in turn benefit wider environmental impact studies such as public health, climate change, environmental intelligence etc. 
The results of this study also serve as a solid stepping stone for spatio-temporal modellers and environmental AI researchers to embark on their research directly with these reanalysis outputs data using stochastic models. 

%With the eligibility of being modelled via stochastic statistical models solely, 
%the utility of such reanalysis data has been widened and 
%the reanalysis data themselves can be further better without being fused with other ground monitoring observations, but instead can be modelled through a multivariate stochastic modelling scheme for instance, the refined result of which will benefit wider environment-impact studies such as public health, climate change, ecosystem, etc.
\end{abstract}


 \section{Introduction}
\label{intro}

This work is motivated by the ECMWF CAMS reanalysis data set. 

The European Centre for Medium-Range Weather Forecasts (ECMWF) implements the Copernicus Atmosphere Monitoring Service (CAMS) on behalf of the European Union \citep{CAMS} and produces the CAMS reanalysis data product,  which is the most updated version of the reanalysis data set of atmospheric compositions including aerosols, e.g., PM25, PM10, chemical species, e.g., sulfate, black carbon, as well as greenhouse gases, e.g., CO2, NO2 etc. on a global scale, consisting of four-dimensional (three spatial dimensions and one temporal dimension) atmospheric composition fields and currently covering 2003 to 2016 (and will be extended one year ahead each year) \citep{inness2019cams}. 
%why such data 

Therefore, the ECMWF CAMS reanalysis data product has become important and resourceful information for researchers in many different environmental-related research areas, such as spatio-temporal modelling, public health, climate change, environmental intelligence etc. to study environmental-related research topics and to support policymakers to address environmental-related concerns as well as the corresponding impacts.

% How to produce such data set? 
Unlike observational measurement data set, which are normally obtained from monitoring equipment and therefore are usually deemed as one realisation from real-world hidden spatial and/or temporal processes with intrinsic randomness, 
the ECMWF CAMS reanalysis data are produced via an \textit{incremental 4D-Var data assimilation methodology} \citep{courtier1994strategy} in which observations such as satellite retrievals of total column CO, aerosol optical depth etc. are assimilated with 12-hour assimilation windows from 09:00 to 21:00 and two spectral truncation T95 (210 km) and T159 (110 km) \citep{inness2019cams}, and are therefore obtained not from monitoring equipment directly but from computers.

Same produced by computers, conventional numerical weather forecast models usually use partial differential equations to model the physical law of the atmospheric compositions \citep{lorenz1963deterministic}
and hence their model outputs are usually deemed deterministic. However, whether the reanalysis outputs that are generated from computers yet via a 4D-Var data assimilation (DA) mechanism are deterministic or stochastic remains largely unclear in the current literature. 

Consequently, the utilities of such computer-generated reanalysis outputs are restricted to either empirical utilities such as climatological computing (e.g., climatological means, percentiles etc.), trend studying, geographical visualisation mapping \citep{inness2019cams},
or limited statistical applications, one is uncertainty quantification and the other is data fusion. %(or downscaling). 

The uncertainty quantification is 
based on assuming the mathematical model executed by the computer is deterministic and viewing it as a ``black box" that makes no use of any of the mathematical information in the model \citep[Section.~1.1]{kennedy2001bayesian}, whereas insisting on the existence of a number of uncertainties in the model outputs. 
\citet[Section.~2]{kennedy2001bayesian} introduced one possible way to subjectively classify different sources of uncertainty under the postulation of the computer model being a ``black box". 

The explicit quantification method is 
through building a statistical model on an \textit{ensemble} of computer model outputs, 
%either a \textit{simple ensemble} which generates each ensemble member by varying initial conditions, (e.g., \citet{sain2011spatial}),
%or \textit{physics-perturbed ensemble} which is obtained by varying model physical parameters (e.g., \citet{murphy2004quantification}),
%or \textit{multi-model ensemble} by collecting a number of different models together (e.g., \citet{smith2009bayesian}).  
and the justification for bringing in statistical models in the presence of deterministic computer model outputs
is because of the limited availability of the members of an ensemble either due to onerous computational time per run or limited availability of different models or funding etc. \citep[Section.~1]{sain2011spatial}.

The central idea of fusion method is to relate the computer model outputs to observational data, see, e.g. \citet[Chp.~5]{kalnay2003atmospheric} and the references therein. One recently-developed fusion method particularly in spatio-temporal research field is the downscaler \citep{berrocal2010spatio}. The downscaler model
relates monitoring observation data $Z(s)$ at location $s$ to computer model outputs $x(B)$ for a specific grid cell $B$ in a linear regression fashion along with an additive Gaussian random measurement error, and their spatially-varying regression coefficients are further jointly modelled as a bivariate Gaussian spatial process via a classical multivariate modelling scheme called the linear model of coregionalization (LMC), see \citet[p.~172]{wackernagel2013multivariate} or \citet[p.~278]{banerjee2014hierarchical}.
For more details on this method see \citet{berrocal2010spatio} or \citet{zidek2012combining}. For the direct application of such a method on the CAMS reanalysis data see \citet{shaddick2020global}.

However, such a fusion modelling method together with observations, to some extent, covers the net stochastic property of the computer outputs data. For instance, in \citet{shaddick2020global}), the readers are unsure about whether the Gaussian stochastic process and random errors in the downscaler models are purely attributed to random monitoring observations or shared by both the monitoring observations and the computer outputs CAMS reanalysis data.

The vague and unknown stochastic property of the computer-generated reanalysis data thus hinders the utility of such data set from wider and more flexible statistical modelling applications. For example, it is not clear whether or not the standard stochastic spatio-temporal modelling framework with additive spatial and temporal random terms, e.g, \textit{Data = Covariates + Spatial random effects + Temporal random effects + (Spatio-temporal interaction) + Random measurement error} \citep[p. 304-305]{cressie2015statistics} can be applied to these 4D-Var reanalysis outputs (i.e., ECMWF CAMS reanalysis data) solely, and whether or not classical multivariate spatial modelling framework (e.g., stochastic co-kriging) can be applied to this CAMS reanalysis data solely without fussing with other monitoring observational data should the joint relationship between different atmospheric compositions within this reanalysis data set are of particular research interest.  

Admittedly, some may argue that from a Bayesian perspective, everything can be random, and we argue that Bayesian’s randomness can only be passively associated with every part (including parameters) of the already-constructed model but is unable to actively justify or indicate the model structure whether to involve spatial and/or temporal random parts or not, which is almost completely guided by the properties and features of the data set themselves. Only when the data set do exhibit spatial and/or temporal stochastic properties, can we involve them in the model. Or from a general error view, only when there are errors induced by spatial and/or temporal sources, can we decompose the general error term of a simple regression model into spatial random effects, temporal random effects, and other measurement errors.

To this end, this paper studies the stochastic property of the reanalysis outputs of 4D-Var. 

To our best knowledge, we are the first to open up the ``black box" of the 4D-Var DA mechanism and then use measure theory to prove 
the existence of spatial and temporal stochasticity associated with the reanalysis outputs. And unlike \citet[Section.~2]{kennedy2001bayesian} who subjectively postulated the possible errors in the computer model outputs, we objectively analysed the different types of error sources as well as their mutual dependence/independence relationship according to the 4D-Var mechanism to complete our understanding towards the stochasticity of the reanalysis outputs and consequently, be able to provide a clear and definite answer to the unsureness about the CAMS reanalysis data in actual modelling practice discussed above. 

In Section 2, we present an exposition of the 4D-Var DA mechanism to open up the ``black box". 
In Section 3, we prepare the readers with concepts relating to stochastic processes and propose our conjecture about the existence of stochasticity in the 4D-Var reanalysis outputs. 
The measure-theory proofs are in Section 4, in which we lay down our proofs from two different perspectives. In Section 5, we focus on the dissection of different types of random errors associated with these reanalysis outputs 
and their mutual dependence/independence properties. 
We end this paper with a discussion including the practical bearings of this research in Section 6.

\section{Mechanism of 4D-Var DA} 
\label{mechanism}

In this section, we crystallize the ``black box" 4D-Var DA.

\citep{bannister2001elementary} 4D-Var DA is a method which finds the best possible initial state $\textbf{x}_{\cdot}$ for every data assimilation cycle. The optimized initial state $\bfx_{\cdot}^A $ (the dot in the subscript here means any iteration cycle) is \textit{de facto} our desired reanalysis output which is obtained from minimizing a cost function $\textit{J}$ measuring misfits between two terms:
one is the discrepancy between an arbitrary state $\textbf{x}_{\cdot}$ at every assimilation cycle
    and a so-called background state $\textbf{x}_{\cdot}^B$,
and the other is the difference between the predicted observations $\textbf{\textit{H}}_{\ttilde} ^o \cdot \textbf{x}_{\ttilde}$ for a given collection time window
    and the real observations $\textbf{y}_{\ttilde}$ collected during this particular time window.  

The cost function $\itJ$ is defined as :
\begin{equation}
    \label{J_fun} %% all following equation citeref  + 1, e.g. eq:2 means equation 3 in paper
    J[\bfx_{k \Delta t}] = \frac{1}{2} (\bfx_{k \Delta t} ^B - \bfx_{k \Delta t})^T \textbf{B}_{k \Delta t} ^{-1} (\bfx_{k \Delta t} ^B - \bfx_{k \Delta t}) 
    + \frac{1}{2} \sum_{\ttilde = k \Delta t} ^{(k + 1) \Delta t} (\bfy_{\ttilde} - \bfH_{\ttilde} ^o \cdot \bfx_{\ttilde})^T \textbf{R}_t^{-1} (\bfy_{\ttilde} - \bfH_{\ttilde} ^o \cdot \bfx_{\ttilde}),
\end{equation} 
where 
\begin{itemize}
    \item $k$: the counter for each run of assimilation cycle; $k = 0, 1, 2, \ldots$; 

    \item $\bfx_{k \Delta t}$: the initial model state for the $(k\Delta t) ^{th}$ run of assimilation, which is a vector collecting the values of atmospheric compositions, e.g., PM25, black carbon (BC), sulfate (SU), etc., across all grid locations at a given time-step $t$, e.g., $k = 0$, then $k \Delta t \triangleq t = 0$. %Thus, these meteorological, atmospheric and chemical component values are sometimes referred to as different fields, e.g., PM25 field, black carbon (BC) field. 
    The optimized value of this initial state, denoted as $\bfx_{k \Delta t}^A$, is the desired reanalysis output result, and is right the CAMS reanalysis product provided by ECMWF;
    
    \item $\bfx^B_{k \Delta t}$: called background state at the $(k\Delta t) ^{th}$ run of assimilation, and is obtained by applying a numerical weather prediction (NWP) model which is a series product of time operators onto the reanalysis output from the last run of assimilation, e.g., $\bfx^B_{\Delta t} =  \bfM_{\Delta t} \cdot \bfM_{\Delta t - \delta t} \cdots \bfM_{\delta t} \cdot \bfM_{0+\delta t} \bfx_0^A$, and here $\delta t$ is incremental time step within the window $ 0 \sim \Delta t$; 
    
    \item $\textbf{B}_{k \Delta t}$: background-state error covariance matrix containing variance and covariance of background-state errors of different atmospheric compositions across different grid locations for a given assimilation time $t = k \Delta t$; 
    Errors can be systematic (e.g., biases) or random. Biases are usually corrected before the assimilation procedure starts and random errors are assumed to be Gaussian;
    When producing CAMS reanalysis product, this error covariance matrix between different fields is block diagonal meaning each field is assimilated univariately \citep[Section.~2.3]{inness2019cams};
    
    \item $\bfy_{\ttilde}$: observations collected during the time window $\ttilde = k \Delta t \sim (k+1)\Delta t$;
    
    \item $\bfx_{\ttilde}$ : model states during the time window $\ttilde = t \sim t + \Delta t$ obtained by evolving from time $\ttilde = t$ via a NWP model, e.g., if $\ttilde = 0 \sim \Delta t$, then $\bfx_{\Delta t} = \bfM_{\Delta t} \cdot \bfM_{\Delta t - \delta t} \cdots \bfM_{2 \delta t} \cdot \bfM_{0+\delta t} \cdot \bfx_0$; Details of the NWP $\bfM \bfM \cdots \bfM$ can be found in Appendix \ref{app:NWP}.
    
    \item $\bfH_{\ttilde} ^o \cdot \bfx_{\ttilde}$: predicted observations obtained via left multiplying a given model state $\bfx_{\ttilde}$ by an observation location smoothing operator matrix $\bfH_{\ttilde} ^o$ which contains interpolation coefficients for different observations at different locations; Details of the structure of $\bfH_{\ttilde} ^o$ see Appendix \ref{app:H};
    
    \item $\bfR_{\ttilde} $ : observations error covariance matrix during the time window $\ttilde$; 
    
    \item Note: subscript is used to express a discrete time step.
\end{itemize}

%Practically, follow the descriptions of \citet{PeterDA19}, data collection, data assimilation (DA) and forecasting are separated into three distinct yet consecutive stages:

%\begin{enumerate}
%    \item data collection stage when observations are collected within a given time window, usually 6-12 hours;
%    \item data assimilation stage when computation of the 4D-Var DA are executed and the outputs of this stage are the desired ECMWF reanalysis product;
%    and then follows
%    \item forecasting stage, which generates forecasting product of ECMWF's. 
%\end{enumerate}

%\citet{PeterDA19} also illustrated these three sequential stages of one cycle of practical operation of data assimilation (DA) with a graph as showed 
%in Figure \ref{fig:4d-var_cycle}.
%\begin{figure}
%    \centering
%    \includegraphics[width=0.85\textwidth]{Figures/Publications_stochastic_4D_Var/schematic_cycle_4dvar.png}
%    \caption{Schematic representation of data collection and assimilation computation of operational DA \citep{PeterDA19}.}
%    \label{fig:4d-var_cycle}
%\end{figure}

%So, combining these practical information, 
So, based on the above, we elaborate below the detailed process of how the ECMWF CAMS reanalysis data product is generated, the procedure of which also provides an excellent opportunity to understand thoroughly the properties of these reanalysis outputs as well as those of their corresponding errors.

In the first run (or time-step), $k = 0$, hence $k\Delta t \triangleq t = 0$, and $\ttilde = 0 \sim \Delta t$, collect observations $\bfy_{\ttilde}$,
\begin{itemize}
    %\item collect observations $\bfy_{\ttilde}$;
    \item background state vector $\bfx^B_0$ at the initial time $t = 0$ is assumed to be a guess;
    \item state $\bfx_0 $ at the initial time $(t = 0)$ is set to be the same as above $\bfx^B_0$;
    \item evolve model state $\bfx_{\ttilde}$ from $\ttilde = t = 0 $ into any required discrete time step within the window $\ttilde = 0 \sim \Delta t$ by $ \bfx_{\ttilde} = \bfM_{\ttilde}\cdot \bfM_{\ttilde - \delta t} \cdots \bfM_{2\delta t} \cdot \bfM_{0+\delta t} \cdot \bfx_0$;
    \item multiply the observation location smoothing operator $\bfH_{\ttilde}^o$ to get the modelled observation $\bfH_{\ttilde} ^o \cdot \bfx_{\ttilde}$; 
\end{itemize}
then by minimizing the cost function $J$ with respect to $\bfx_0$ for this run, we get the first optimized reanalysis output for the initial state of this run and denote it as $\bfx^A _0$ = arg min $J$ to replace the initial guess $\bfx_0$. 

In the second run, $k = 1$, hence $t = \Delta t$, and $\ttilde = \Delta t \sim 2 \Delta t$, collect observations $\bfy_{\ttilde}$,
\begin{itemize}
    %\item collect observations $\bfy_{\ttilde}$; 
    \item the initial time of this run is $t = \Delta t$, and so the background state vector is $\bfx^B_{\Delta t}$, which is obtained by applying an NWP model onto the reanalysis output $\bfx_0 ^A$ from the last run, that is  
    $\bfx^B_{\Delta t} = \bfM_{\Delta t} \cdot \bfM_{\Delta t - \delta t} \cdots \bfM_{0+\delta t}\cdot \bfx_0 ^A $;
    \item model state for this run is $\bfx_{\ttilde (= \Delta t \sim 2 \Delta t)} = \bfM_{\ttilde}\cdot \bfM_{\ttilde - \delta t} \cdots \bfM_{\Delta t + 2\delta t} \cdot \bfM_{\Delta t +\delta t} \cdot \bfx_{\Delta t}$; 
    \item multiply the observation location smoothing operator $\bfH_{\ttilde}^o$ to get the modelled observation $\bfH_{\ttilde} ^o \cdot \bfx_{\ttilde}$ for this run; 
\end{itemize}
and minimize the cost function $J$ with respect to model state $\bfx_{\Delta t}$ we then get the second optimized reanalysis output for the initial state of this run and denote it as $\bfx^A_{\Delta t}$ = arg min $J$.

%In the third run, $t = 2 \Delta t \sim 3 \Delta t$,
%\begin{itemize}
%    \item collect observations $\bfy_{\ttilde}$;
 %   \item the initial time of this run is $t = 2\Delta t$, so the background state vector is $\bfx^B_{2 \Delta t} = M_{2\Delta t} \cdot M_{2\Delta t - \delta t} \cdots M_{\Delta t + \delta t}\cdot \bfx_{\Delta t} ^A $; 
%    \item model state for this run is $\bfx_{t (= 2\Delta t \sim 3 \Delta t)} = M_{t}\cdot M_{t - \delta t} \cdots M_{\Delta t + 2\delta t} \cdot M_{\Delta t +\delta t} \cdot \bfx_{2\Delta t}$; 
%    \item multiply the observation operator $H_t^o$ to get the modelled observation $H_t ^o \cdot \bfx_t$ for this run; 
%\end{itemize}

%again, minimizing $J$ with respect to $\bfx_{2 \Delta t}$ to get the third optimized reanalysis output for the initial state of this run, which is $\bfx^A_{2\Delta t}$ = arg min $J$. 

Following the same logic, we can get a collection of the reanalysis outputs $ \{ \bfx^A _0, \bfx^A_{\Delta t}, \bfx^A_{2\Delta t}, \ldots   \}$.

%We summarize and illustrate the 
%above expositions of the procedure for obtaining the reanalysis outputs 4D-Var DA in Figure %\ref{fig:4Dvar_cycle} 
%\ref{appfig:4Dvar_cycle} in appendix. 
Illustrative derivation of the cost function $J$ for the first two runs as well as their corresponding first derivative expressions (for arg min J) are detailed in Appendix \ref{app:J&DJ}, especially in equation \eqref{nablaJ}. In general, the reanalysis output $\bfx_t \triangleq \bfx_t^A$ at time $t$ is in the form of 
\begin{align}
    \label{eq:bfx}
    \bfx_t \triangleq \bfx_t^A \propto L \bfx_t^B + \sum_{\ttilde = t}^{t + \Delta t} K_{\ttilde} \bfy_{\ttilde}, 
\end{align}
where $L$ and $K_{\ttilde}$ are two coefficient matrices consisting of $\bfH \bfM \cdots \bfM$, and in particular, the number of terms of $\bfM$ involved in $K_{\ttilde}$ depends on $\ttilde$. Derivation details see equation \eqref{eq:bfx_tA} and \eqref{eq:bfx_tA_compact} in Appendix \ref{app:J&DJ}.

\section{Conjecture of the Existence of Stochasticity in the Reanalysis Outputs}
In this section, we lay down our conjecture about the existence of stochasticity or randomness in
the above reanalysis outputs $ \{ \bfx^A _0, \bfx^A_{\Delta t}, \bfx^A_{2\Delta t}, \ldots  \}$, or equivalently the ECMWF CAMS reanalysis data. 

To start with, we first briefly review some basic concepts relating to stochastic processes. We follow the conventions in \citet[p.~360]{grimmett2001probability} and \citet[p.~3-5]{cosmaadvprob2}.

\subsection{Concepts of Stochastic Processes}
\label{stoch-proc}
A \textit{stochastic process} $\bfX$ is a collection or a family $\{ \bfX_t : t \in T \}$ of random variables $X_t$'s or generally random objects \footnote{When the $sigma$-algebra of the output space is in $\mbbR^1$, it's a random variable; when the $sigma$-algebra of the output space is in $\mbbR^n$, it's a random vector, and when the $sigma$-algebra of the output space is a sequence, it's a stochastic (random) process.} $\bfX_t$'s, each of which maps a sample space $\Omega$ into a state space $U$, that is $X_t : \Omega \rightarrow U$ or $\bfX_t : \Omega \rightarrow U$.

The index set $T$ can be discrete e.g. $\{0, 1, 2, \ldots \}$ or continuous e.g. $[0, \infty)$, and the state space $U$ can be integer $\mathbb{Z}$ or real number $\mathbb{R}$. Together, the choice of the index set $T$ and state space $U$ decides the analytic property of stochastic processes. 

For any fixed sample element $\omega \in \Omega$, there's a corresponding collection $\{ \bfX_t (\omega) : t \in T \} \subset U$, which is called  \textit{one realization} or \textit{a sample path} of the stochastic process $\bfX$ at $\omega$. 

When the length of the index set $T$ is 1, that is the index set $T$ has only one element, then the stochastic process degenerates into a \textit{trivial stochastic process}, or equivalently a random variable.

\subsection{Stochasticity-Existence Conjecture}
\label{exam-stoch}
%% assume the input value is appropriate (e.g. not some extreme values that model is solvable)
%% x(t) is for convenient in place of x_t

Now come back to the collection of reanalysis outputs $ \{ \bfx^A _0, \bfx^A_{\Delta t}, \bfx^A_{2\Delta t}, \ldots   \}$ from the 4D-Var DA, and from the equation \eqref{J_fun} and \eqref{eq:bfx}, we know the reanalysis output $\bfx_t^A$ at each run $t$ is essentially a function of background state $\bfx_t^B$ at time $t$ and observations $\bfy_{\ttilde}$ collected from $t$ to $t + \Delta t$. 
That is $\bfx^A _t = $ arg min $J$ = $\mfkG(\bfx^B_t, \bfy_{\ttilde})$. 

Although the background state $\bfx_t^B$, obtained from the reanalysis output of the last run, whose randomness is still to be proved (except the initial run in which $\bfx_t^B$ is set to a guess) and hence can be treated as a constant for the moment (more discussion on this see Section \ref{discussion}), the observations $\bfy_{\ttilde}$ collected from $t$ to $t + \Delta t$ are random in nature, since they are usually viewed as one sample from a hidden spatial and/or temporal process, meanwhile the observations also contain random errors, e.g., measurement error etc. This means, on one hand, $\bfx^A _t = $ arg min $J$ = $\mfkG(\bfx^B_t, \bfy_{\ttilde})$ can be further simplified to $\bfx_t^A = \mcalG(\bfy_{\ttilde})$, on the other, the observations $\bfy_{\ttilde}$ fed into the cost function J or equation \eqref{eq:bfx} are just sample elements $\omega_{\bfy_{\ttilde}} \triangleq \bfy_{\ttilde}$ from their sample space $\Omega_{\bfy_{\ttilde}}$ ($\ttilde = t \sim (t + \Delta t)$), and therefore, we could write $\bfx_t^A = \mcalG(\bfy_{\ttilde})$ in a more general way which reflects the random nature of $\bfy_{\ttilde}$ as 
\begin{align}
    \label{eq:G}
    \bfX^A_t = arg\: min\: J = \mcalG(\bfY_{\ttilde})
\end{align}
almost surely, where here the upper cases denote random variables while the lower cases denote their corresponding realisations. And by equation \eqref{eq:bfx}, we know equation
\eqref{eq:G} can be further written out in a more explicit form as 
\begin{align}
    \bfX^A_t = arg\: min\: J = \mcalG(\bfY_{\ttilde}) \propto C_t + \sum_{\ttilde = t} ^{t + \Delta t}K_{\ttilde} \bfY_{\ttilde}, 
\end{align}
where $C_t$ denotes a constant vector and $K_{\ttilde}$ is the coefficient matrix consisting of $\bfH \bfM \ldots \bfM $.
In this way, we obtain what will be referred to as \textit{the general form of the reanalysis output} $\bfX_t^A$, where $t = 0, \Delta t, 2 \Delta t, \ldots$.

Consequently, by the concepts introduced in Section \ref{stoch-proc}, 
we know that each run of the reanalysis output 
$\bfx^A _t = \bfX^A _t (\omega_{\bfy_{\ttilde}})$ is one realisation from a trivial stochastic process $\bfX^A _t$ ($t = 0, \Delta t, 2\Delta t, \ldots$). 
Hence, collectively, although we can not confer the collection of the reanalysis outputs $\{ \bfX^A _0 (\omega_{\bfy_{\ttilde (= 0 \sim \Delta t)}}),
\bfX^A _{\Delta t} (\omega_{\bfy_{\ttilde (= \Delta t \sim 2 \Delta t)}}), \bfX^A _{2 \Delta t} (\omega_{\bfy_{\ttilde (= 2 \Delta t \sim 3 \Delta t)}}), \ldots \}$
directly as one realisation of a temporally evolved stochastic process, which would require a fixed element in the sample space (e.g. $\omega_{\bfy_{\ttilde (= 0 \sim \Delta t)}}$) instead of sample elements that change along the time index (i.e. $\omega_{\bfy_{\ttilde (= 0 \sim \Delta t)}}, \omega_{\bfy_{\ttilde (= \Delta t \sim 2 \Delta t)}}, \ldots$)
across all the temporally indexed
random variables $\bfX^A_t$ ($t = 0, \Delta t, 2\Delta t, \ldots$),
we can at least have a collection of different realisations from each trivial stochastic process. 

And from here, we may be able to further explore the possibility of the equivalence of this collection of different realisations from each trivial stochastic process and one realisation or a sample path from a temporal stochastic process, which is more commonly seen in the spatio-temporal stochastic modelling realm.

And if we just focus on one reanalysis output at a given time $t$, 
$\bfx^A_t \equiv \bfX^A_t (\omega_{\bfy_{\ttilde (= t \sim (t+\Delta t))}})$, %\equiv
%\bfX^A _t (\bfy_{\ttilde (= t \sim (t+\Delta t))})$, 
from the mechanism of 4D-Var DA in Section \ref{mechanism},
we know that this reanalysis output is a vector collecting all the atmospheric field values
across all the spatial grid locations within a potentially infinite yet practically finite domain $\{s_1, s_2, s_3, \ldots, s_N \} \subset \{s_1, s_2, s_3, \ldots \}$.
That is at a given time $t$, we have a spatially evolved collection 
$\{ \bfX^A _{s_1} (\omega_{\bfy_{\ttilde}}; t), 
\bfX^A _{s_2} (\omega_{\bfy_{\ttilde}}; t),
\ldots \bfX^A _{s_N} (\omega_{\bfy_{\ttilde}}; t) \} $, 
which is a sub-collection of
$\{ \bfX^A _{s_1} (\omega_{\bfy_{\ttilde}}; t), 
\bfX^A _{s_2} (\omega_{\bfy_{\ttilde}}; t),
\ldots \} $, $\ttilde = t \sim (t + \Delta t)$.
%and similarly, at another fixed time, we have
%$\{
%\bfX^A _{s_1} (t = \Delta t, \bfx^B_{\Delta t}, \bfy_{\ttilde}, H_t^o \bfx_t),
%\bfX^A _{s_2} (t = \Delta t, \bfx^B_{\Delta t}, \bfy_{\ttilde}, H_t^o \bfx_t),
%\ldots, 
%\bfX^A _{s_N} (t = \Delta t, \bfx^B_{\Delta t}, \bfy_{\ttilde}, H_t^o \bfx_t) \} 
%$
%and so on. 
Such a collection of reanalysis outputs evolving across different spatial indices
at a given time looks very similar to one realisation from a \textit{spatial} stochastic process at that given time $t$.

%Furthermore, if we not only choose a given time $t$, but also choose a given grid location $s_i$ at this time, we know $\bfx^A_{s_i} (t) = \bfX^A_{s_i} (t, \bfx^B_t, \bfy_{\ttilde}, H_t^o \bfx_t)$ collects all the analysis values for all the atmospheric aerosol fields, e.g., black carbon, dust, sulfate, etc.,
%that is $\bfx^A _{s_i} (t) = \bfX^A _{s_i} (t, \bfx^B_t, \bfy_{\ttilde}, H_t^o \bfx_t) = [x^{A_{BC}} _{s_i} (t, \bfx^B_t, \bfy_{\ttilde}, H_t^o \bfx_t), 
%x^{A_{DU}} _{s_i} (t, \bfx^B_t, \bfy_{\ttilde}, H_t^o \bfx_t), 
%x^{A_{SU}} _{s_i} (t, \bfx^B_t, \bfy_{\ttilde}, H_t^o \bfx_t)
%\ldots]^T$. 
%Here \textit{BC} represents black carbon, \textit{DU} represents dust, and \textit{SU} represents sulfate.

\section{Proofs of Existence of Stochasticity}

Thus far, our analyzing work is mainly based on the concepts relating to the stochastic process. To rigorously demonstrate the existence of stochasticity associated with the reanalysis outputs of 4D-Var DA (i.e., the ECMWF CAMS reanalysis data), we resort to measure theory.

In measure theory, stochastic processes can be constructed via two paths: one is through a collection of random variables (or random objects) defined on a common probability space, and the other is via an abstract dynamical system \citep[p.~6-7]{gray2009probability}. We now prove the existence of stochasticity associated with our reanalysis outputs from each of these two perspectives.

\subsection{Perspective 1: A Sequence of Random Variables}
\label{MT-seqrdv}
Section \ref{stoch-proc} states that a stochastic process is a collection of random objects (random variables or random vectors, or stochastic processes). But we need to rethink what a \textit{random object} really is and whether \textit{the general form of the reanalysis output} $\bfX^A_t$ $(t = 0, \Delta t, 2 \Delta t, \ldots)$ we defined in Section \ref{exam-stoch} indeed matches the concept of a \textit{random object}. We first lay down the basic theories and then examinations and proofs follow. 

By the definition in \citet[p.~4]{gray2009probability}, \citet[p.~182-186]{billingsley1995probability}, and \citet[p.~192]{athreya2006measure}, 
given two measurable spaces $(\Omega, \mcalA)$ and $(\mbbR^n, \mcalB(\mbbR^n))$, where $\Omega$ denotes a sample space, $\mcalA$ is a $\sigma$-algebra on $\Omega$, $\mbbR^n$ is a topological space, and $\mcalB(\mbbR^n)$ is a Borel $\sigma$-algebra of subsets of $\mbbR^n$, a function $f: \Omega \rightarrow \mbbR^n$ is 
a \textit{random vector} or its one-dimensional special case \textit{random variable} (n = 1) if it is a \textit{measurable function} on $(\Omega, \mcalA)$, that is 
\begin{align*}
    f^{-1}(B) = \cap_{i = 1}^{n} \{\omega: f_i(\omega) \in B  \} \in \mcalA, \; \forall B \in \mcalB(\mbbR^n), 
\end{align*}
and if $\mu$ is a measure on $\mcalA$, that is if $(\Omega, \mcalA, \mu)$ is a measure space, then there exists an \textit{induced measure} $\mu f^{-1}$ on $\mcalB(\mbbR^n)$ such that
\begin{align*}
    \mu f^{-1} (B) = \mu (f^{-1} (B)) = \mu (\cap_{i = 1}^{n} \{\omega: f_i(\omega) \in B  \}), \; \forall B \in \mcalB(\mbbR^n)
\end{align*}
which is induced by the random object $f$. This induced measure is called \textit{(joint) cumulative density function} for the random object $f$, which is itself non-decreasing, right-continuous, hence is a Lebesgue-Stieltjes measure. 

And if $f: \mbbR^i \rightarrow \mbbR^k$ mapping between two topological spaces is continuous, then $f$ is \textit{Borel} measurable, see the Theorem 3.2 in \citet[p.~183]{billingsley1995probability}.

By the analysis in Section \ref{exam-stoch}, we know $\bfX_t^A = \mcalG(\bfY_{\ttilde}) \propto C_t + \sum_{\ttilde = t} ^{t + \Delta t}K_{\ttilde} \bfY_{\ttilde} $, where $\bfY_{\ttilde}$ is an $n \times 1$ vector of observational variables, and here $n$ is the number of grid locations $N$ times the number of atmospheric compositions $P$, $K_{\ttilde}$ is an $n \times n$ scalar matrix consisting of $\bfH \bfM \ldots \bfM$, in which the number of terms of $\bfM$ involved depends on index $\ttilde$,
and $C_t$ is an $n \times 1$ vector of constants, for
detailed derivation, see Appendix \ref{app:J&DJ}. So,
to verify whether the general form of the reanalysis output $\bfX^A_t$ $(t = 0, \Delta t, 2 \Delta t, \ldots)$ is a random object or not, we just need to prove that the function $\mcalG(\cdot)$ is a measurable one. 

\begin{lemma}
\label{lemma 4.1}
The general form of the reanalysis outputs $\bfX^A_t (t = 0, \Delta t, 2\Delta t, \ldots)$ from 4D-Var DA is a random object. 
\end{lemma}

\begin{proof}
We first denote $\bfZ_{\ttilde} = K_{\ttilde} \bfY_{\ttilde} = G_{\ttilde}(\bfY_{\ttilde}) $, where $\ttilde \in [t, (t+\Delta t)]$ and $G_{\ttilde}(\cdot)$ is a function,
and so $\bfX_t^A = \mcalG(\bfY_{\ttilde}) \propto C_t + \sum_{\ttilde = t} ^{t+\Delta  t}K_{\ttilde} \bfY_{\ttilde} = C_t + \sum_{\ttilde = t} ^{t+\Delta t} G_{\ttilde}(\bfY_{\ttilde})$. Since measurable functions are closed under addition, scalar translation and scalar multiplication, so, to prove $\bfX_t^A = \mcalG(\cdot)$ is measurable we just need to prove $\bfZ_{\ttilde} = G_{\ttilde}(Y_{\ttilde})$ for any one index $\ttilde \in [t, t+\Delta t]$ is measurable. Without loss of generality, we choose $\ttilde = t$.

And by the definition of $\bfY_{\ttilde}$ and $\bfZ_{\ttilde}$, we know $G_{\ttilde = t}(\cdot)$ is a function mapping between two topological spaces, i.e., $G_{\ttilde = t}: \mbbR ^n \rightarrow \mbbR^n$, with each of the topological spaces being equipped with \textit{Borel} \sigalg \: $\mcalB (\mbbR^n)$ generated by the Cartesian product of $n$ open intervals, and in particular, the $\sigma$-algebra $\mcalB (\mbbR^n)$ on the output space is $\mcalB (\mbbR^n) = \sigma \langle (-\infty, b_1)\times \ldots \times (-\infty, b_n) \rangle$, $(b_i: i = 1, \ldots n) \in \mbbR^n$. And since $G_{\ttilde = t}(\bfY_{\ttilde = t}) = K_{t} \bfY_{t}$ is a linear map, hence is continuous, therefore
it is \textit{Borel} measurable, by Theorem 3.2 in \citet[p.~183]{billingsley1995probability}. This means 
\begin{align*}
    G_{\ttilde = t}^{-1} ((-\infty, b_1)\times \ldots \times (-\infty, b_n)) = \cap_{i = 1} ^n \{ \omega_{y_t}: [\ G_{\ttilde = t} (\omega_{y_t}) ]\ _i \in (- \infty, b_i) \}
\end{align*}
is \textit{Borel} measurable. Hence $\bfX_t^A = \mcalG (\bfY_{\ttilde}) \propto C_t + \sum_{\ttilde = t} ^{t+\Delta  t}K_{\ttilde} \bfY_{\ttilde} = C_t + \sum_{\ttilde = t} ^{t+\Delta t} G_{\ttilde}(Y_{\ttilde}) $ is \textit{Borel} measurable by each $G_{\ttilde}(Y_{\ttilde})$, $\ttilde \in [t, t+\Delta t]$ is \textit{Borel} measurable. $\bfX_t^A$ is a random object. 

\end{proof}

From the above proof, we know either $\ttilde$ being just one index (e.g., $\ttilde = t$) or spanning across an interval (e.g., $t \sim t+\Delta t$) does not change the measurability of $\mcalG$ and hence the randomness of $\bfX_t^A$, so it's notationally clear and convenient to just choose one index for $\ttilde$, and without loss of generality, we set $\ttilde = t$, so $\bfX_t^A = \mcalG(\bfY_{\ttilde = t})= G_{\ttilde = t}(\bfY_{\ttilde = t}) + C_t$, and denote the sample space for $\bfY_t$ as
 $\Omega$ with equipped \textit{Borel} $\sigma$-algebra $\mcalA$, and one
sample element $\omega_{\bfy_{\ttilde = t}}$ of $\bfY_{\ttilde = t} $ will be denoted as $\omega_t$ for convenience, so 
in the following article, we work with 
\begin{align*}
    \bfX_t^A = \mcalG(\bfY_t), 
\end{align*}
where $\mcalG : (\Omega, \mcalA) \rightarrow (\mbbR^n, \mcalB(\mbbR^n))$ with 
the corresponding realisation of $\bfX_t^A$ being $\bfx_t^A = \bfX_t^A (\omega_t) = \mcalG(\omega_t)$. 

And if there is a probability measure $P$ on $(\Omega, \mcalA)$, there must be an induced probability measure $P\mcalG^{-1}$ on $(\mbbR^n, \mcalB(\mbbR^n))$ induced by the measurable function $\mcalG$.

\begin{lemma}
\label{lemma 4.2}
The collection of the general form of the reanalysis outputs $\{ \bfX^A_t : t = 0, \Delta t, 2\Delta t, \ldots \}$ from 4D-Var DA is a single-sided temporal stochastic process.
\end{lemma}

\begin{proof}
    By Lemma \ref{lemma 4.1}, each $\bfX^A_t$ is a random object defined on a common probability space $(\Omega, \mcalA, P)$, and index $t$ spans across different positive time-steps from 0. 
\end{proof}

Such a temporal stochastic process is essentially a digitised version of a real-world hidden temporal stochastic process.
By the relationship $\bfx^A_t = \bfX^A_t(\omega_t)$, we know the collection of each run of our reanalysis output $\{\bfx^A_0, \bfx^A_{\Delta t}, \bfx^A_{2 \Delta t}, \ldots \} 
= \{ \bfX^A_0(\omega_0), \bfX^A_{\Delta t}(\omega_{\Delta t}), \bfX^A_{2\Delta t}(\omega_{2\Delta t}), \ldots \}$ 
is a collection of individual realisations from each trivial stochastic process or every single random variable. %that consists of the temporal stochastic process. 
In the Theorem \ref{cams-output} of Section \ref{dynamicsys}, we will reveal its equivalence to one realisation or a sample path of the temporal stochastic process.

Further, at a given time $t$, by the definition of $\bfx^A_t$ (or $\bfx_{k \Delta t}$) introduced in Section \ref{mechanism}, which is a vector collecting all the values of atmospheric fields across all grid locations,
we can further expand the random object $\bfX_t ^A$ at the given time $t$ into a ``deeper" collection according to a spatial index set $\mcalS$, 
that is $\{ \bfX_{t, s_1}^A, \bfX_{t, s_2}^A, \bfX_{t, s_3}^A, \ldots \}$,
where here the spatial index set $\{s_1, s_2, \ldots\}$ is theoretically infinite while practically finite with each $s_i \in \mbbR^d$. 
And since $\bfX_t^A$ is a measurable function of $\bfY_{t}$, 
it's trivial to show that $\bfX^A_{t,s_i}$ (\,$i = 1, 2, \ldots$ \,) is also a measurable function of $\bfY_{t}$, 
hence the collection of these random vectors $\{ \bfX^A_{s_i}(t)$, $i = 1, 2, \ldots \}$ with the evolutionary spatial indices at a given time $t$ %(for all atmospheric compositions 
is a \textit{d-dimensional spatially-discrete random field} %(for each of the atmospheric compositions), 
or just simply a \textit{spatial random field} or a \textit{spatial stochastic process} \citep[p.~4]{cosmaadvprob2}. One for each atmospheric composition. Here moving the time index $t$ into a bracket is for a clear observation of the evolution of spatial indices.

\begin{lemma}
\label{lemma 4.3}
At a given time $t$, the collection of the general form of the reanalysis outputs $\{ \bfX^A_{s_i}(t)$, $i = 1, 2, \ldots \}$ of 4D-Var DA is a spatial stochastic process.
\end{lemma}

\begin{proof}
By Lemma \ref{lemma 4.1} and the definition of stochastic process. Details see appendix \ref{app:prooflemma4.3}.
\end{proof}

Such a spatial stochastic process is essentially a digitised version of a real-world hidden spatial process. 
When a finite collection of spatial index set $\{s_1, s_2, \ldots, s_N \} \subset \{s_1, s_2, \ldots \} \subset \mcalS$ is selected, 
then $\{ \bfX^A_{s_1}(t), \bfX^A_{s_2}(t), \ldots, \bfX^A_{s_N}(t) \}$ is a sub-collection of the spatial stochastic process. 

And notice at the given time $t$, the sample element $\omega_t \in \Omega$ is then fixed, 
follow the definition of \textit{one realisation} %or \textit{a sample path} 
of a stochastic process in Section \ref{stoch-proc}, we understand that our reanalysis outputs at this time
$\{\bfx^A_{s_1}(t), \bfx^A_{s_2}(t), \bfx^A_{s_3}(t), \ldots \} \equiv \{ \bfX^A_{s_1}(\omega_t; t), \bfX^A_{s_2}(\omega_t; t), \bfX^A_{s_3}(\omega_t; t)\dots \}$ is \textit{de facto} one realisation %or a sample path 
of the spatial stochastic process. 

Therefore we have arrived at an important result connecting the reanalysis outputs and the realisations of the spatial stochastic process, and we summarize it into below Theorem \ref{sp-stochastic_pro}:

\begin{theorem}
\label{sp-stochastic_pro}
At a given time $t$, the reanalysis outputs $\bfx^A_t = \{\bfx^A_{s_1}(t), \bfx^A_{s_2}(t), \bfx^A_{s_3}(t), \ldots \}$ from 4D-Var DA (i.e., the ECMWF CAMS reanalysis data) is one realisation from a spatial stochastic process 
$\{ \bfX^A_{s_1}(t), \bfX^A_{s_2}(t), \bfX^A_{s_3}(t) \ldots  \}$.
\end{theorem}

\begin{proof}
By Lemma \ref{lemma 4.3} and the definition of one realisation of a stochastic process. Details see appendix \ref{app:proofthrm4.1}.
\end{proof}

%Moreover, by the definition of $\bfx^A_t$ (or $\bfx_{k \Delta t}$) introduced in Section \ref{mechanism} again, at a given time $t$ and a given grid location $s_i$ with the given $\Omega = [\bfx^B_t, \bfy_{\ttilde}, H_t^o \bfx_t]^T$, 
%we could further expand the realisation 
%$\bfX^A_{s_i}(t,\omega_t) = \bfx^A_{s_i}(t), i = 1, 2, 3, \ldots$ as a vector collecting all the atmospheric aerosol fields, i.e., $\bfx^A_{s_i}(t) 
%= [x^{A_{BC}} _{s_i} (t), 
%x^{A_{DU}} _{s_i} (t), 
%x^{A_{SU}} _{s_i} (t)
%\ldots]^T$,
%here \textit{BC} represents black carbon, \textit{DU} represents dust, and \textit{SU} represents sulfate.
%The non-boldface $x$'s indicate they are scalars in stead of vectors. It's possible to think of such realisations, e.g., $x^{A_{BC}} _{s_i} (t)$ being from a random variable $X^{A_{BC}}_{s_i} (t, \omega_t)$, but we don't conclude directly that the collection of such random variables consists of a stochastic process due to the indices here, e.g. $BC$, $DU$ are neither numerical nor ordered naturally unless subjectively postulated so, therefore without loss of generality, we don't claim the existence of such stochastic process but only deal with the realisations of each of such random variables.

%The visualisation of the whole level-structure of different processes in the reanalysis outputs data can be found in Figure %\ref{fig:sp_t_process} 
%\ref{appfig:sp_t_process} in Appendix. 

\subsection{Perspective 2: Abstract Dynamical System}
\subsubsection{Temporal Stochastic Process}
\label{dynamicsys}
From \citet[p.~7]{gray2009probability}, an \textit{abstract dynamical system} consists of two ingredients: one is a measure space $(\Omega, \mcalA, P)$ and the other is a measurable and invertible transformation function $T$ defined on this space as $T: \Omega \rightarrow \Omega$. 
Together, the quadruple $(\Omega, \mcalA, P, T)$ consists of an \textit{abstract dynamical system} that reflects the long-term dynamic behaviour of repeated applications of measurable transformation $T$ on the measure space $(\Omega, \mcalA, P)$. And since the composition of a measurable function is also measurable, $T^r$ defined as $T^r \omega = T(T^{r-1} \omega) = T \circo T^{r-1} (\omega)$ is also measurable with respect to the measurable space $(\Omega, \mcalA)$.

To link the \textit{abstract dynamical system} with a stochastic process, we also need a V-valued\footnote{When the output space of a measurable function is not explicitly stated in the context, we denote it as a \textit{V-valued} random variable. \citep[p.~2]{gray2009probability}} measurable function or a random variable $f$ and assume it is defined on the same measurable space $(\Omega, \mcalA)$, then the new measurable function $fT^r : \Omega \rightarrow V$ 
defined as $fT^r(\omega) = f(T^r\omega)$, $\omega\in \Omega $, is obviously a random variable for all $r \in \mbbZ^+$.

This means an \textit{abstract dynamical system} $(\Omega, \mcalA, P, T)$ and a measurable function $f$ together define a one-sided random process $\{X_r: r \in \mbbZ^+ \}$ with each $X_r(\omega) = fT^r (\omega) = f(T^r \omega)$, 
that is the r$^{th}$ realisation value of this random process (i.e., $X_r(\omega)$) is the value of the measurable function $f$ evaluated at an $r$-unit transformed point in the original sample space $\Omega$ (i.e., $f(T^r \omega)$). This will be a critical property for us to equate one realisation from a temporal stochastic process to a collection of different realisations from every single random variable or trivial stochastic process, hence bringing forward the conclusion about the temporal stochastic process we arrived at in Section \ref{MT-seqrdv}.

When $\Omega$ contains a sequence $\{\omega_r: r \in \mbbZ^+ \}$, $T^l$ shifts the sequence $\{\omega_r: r \in \mbbZ^+ \}$ to the sequence $\{\omega_{r+l}: r \in \mbbZ^+ \}$ where each coordinate is shifted to the left by $l$ time units. 

And if the measurable space $(\Omega, \mcalA) $ is a topological space $ (\mbbR^n , \mcalB(\mbbR^n))$, then Krylov–Bogolyubov theorem ensures on $\mcalB(\mbbR^n)$ there exists an invariant Borel probability measure $\mu$ s.t.  $\mu(T^{-1}(B)) = \mu(B)$, or $T\mu = \mu$ for any $B \in \mcalB(\mbbR^n)$ \citep[p.~4]{sinai1989dynamical}.

Now apply the above theory to our reanalysis outputs and we conclude the relationship between the collection of reanalysis outputs of 4D-Var DA and the realisations of a temporal stochastic process with the following Theorem \ref{cams-output}:
\begin{theorem}
\label{cams-output}
The collection of the reanalysis outputs $\{ \bfx^A_0, \bfx^A_{\Delta t}, \bfx^A_{2 \Delta t}, \ldots \}$ from 4D-Var DA (i.e., the ECMWF CAMS reanalysis data) is indeed one realisation or a sample path of a temporal stochastic process $\{ \bfX_t: t = 0, \Delta t, 2\Delta t, \ldots \}$.
\end{theorem}

\begin{proof}
From the last section \ref{MT-seqrdv}, we know that the general form of reanalysis output is $\bfX_t^A = \mcalG(\bfY_t)$, where $\mcalG$ is a Borel measurable function mapping from $(\Omega, \mcalA)$ to $(\mbbR^n, \mcalB(\mbbR^n))$, and here $(\Omega, \mcalA) = (\mbbR^n, \mcalB(\mbbR^n))$ by the definition of $\bfY_t$ and $\bfX_t^A$. And let $T^{\Delta t}$ shift the coordinates of sequence to the left by $\Delta t$ unit(s), i.e., $T^{\Delta t}(\omega_0, \omega_{\Delta t}, \omega_{2 \Delta t}, \ldots) = (\omega_{\Delta t}, \omega_{2 \Delta t}, \omega_{3 \Delta t} \ldots)$, $T^{2\Delta t}(\omega_0, \omega_{\Delta t}, \omega_{2 \Delta t}, \ldots) = T^{\Delta t}(T^{\Delta t}(\omega_{\Delta t}, \omega_{2 \Delta t}, \omega_{3 \Delta t} \ldots)) = (\omega_{2 \Delta t}, \omega_{3 \Delta t} \ldots))$, and one commonly seen example is $\Delta t = 1$. 

For the first time-step, i.e., $t = 0$, we have $\bfx_0^A = \bfX_0^A(\omega_0)$, and for the second time-step $t = \Delta t$, we have
\begin{equation}
  \begin{aligned}[b]
    \bfx^A_{\Delta t} &= \bfX^A_{\Delta t}(\omega_{\Delta t}) \qquad \mbox{(by definition of one realisation of a random variable (r.d.v.))} \\
    %&\triangleq \bfX_{\Delta t}(\omega) \qquad \mbox{($\triangleq$ means ``denoted as")}\\
    &= \mcalG (\omega_{\Delta t}) \qquad \mbox{(by definition of $\bfX^A_{\Delta t}$ and $\mcalG$ is a measurable function)}\\ 
    &= \mcalG (T^{\Delta t} \omega_0) \qquad \mbox{(by definition of $T$, $T^{\Delta t} \omega_0 = \omega_{\Delta t} $)}\\ 
    &= \mcalG T^{\Delta t} (\omega_0) \qquad \mbox{(by associative property of multiplication)}\\
    &= \bfX_{\Delta t}^A (\omega_0) \qquad \mbox{(by constructing a r.d.v. using dynamical system and a measurable function $\mcalG$)}\\ 
\end{aligned}  
\label{eq:3}
\end{equation}

And similarly, we have 
\begin{align}
    \bfx^A_{2 \Delta t} = \bfX^A_{2 \Delta t}(\omega_{2 \Delta t}) = \mcalG (\omega_{2 \Delta t}) = 
    \mcalG (T^{2 \Delta t} \omega_0) = \mcalG T^{2 \Delta t} (\omega_0)
    = \bfX_{2\Delta t}^A (\omega_0),
    \label{eq:4}
\end{align} 
and 
%\begin{align}    
    %\bfx^A_{3 \Delta t} = \bfX^A_{3 \Delta t}(\omega_{3 \Delta t}) = \mcalG (\omega_{3 \Delta t}) %= 
    %\mcalG (T^{3 \Delta t} \omega_0) = \mcalG T^{3 \Delta t} (\omega_0)
    %= \bfX_{3 \Delta t}^A (\omega_0),
    %\label{eq:5}
%\end{align}
so on. 

We know that a dynamical system $(\Omega, \mcalA, P, T)$ together with a measurable function $\mcalG$ defines a random variable, so the collection $\{ \bfX_{0}^A,
\bfX^A_{\Delta t}, \bfX^A_{2 \Delta t}, \ldots \}$ is a single-sided temporal stochastic process, and by the definition of \textit{one realisation} or \textit{a sample path} of a stochastic process introduced in Section \ref{stoch-proc}, 
we know for a given fixed $\omega_0$, the collection of the rightmost side of the equations \eqref{eq:3}, \eqref{eq:4}, %\eqref{eq:5}, 
i.e.,   
$\{\bfX^A_{0}(\omega_0), \{\bfX^A_{\Delta t}(\omega_0), \bfX^A_{2 \Delta t}(\omega_0), \ldots\}$ is indeed one realisation or a sample path from this single-sided temporal stochastic process. On the other hand, the collection of the leftmost side of equations \eqref{eq:3}, \eqref{eq:4} %\eqref{eq:5} 
is our reanalysis outputs data $\{ \bfx_{0}^A, \bfx_{\Delta t}^A, \bfx_{2 \Delta t}^A, \ldots \}$. Therefore, a collection of our reanalysis outputs from 4D-Var DA is indeed one realisation of the single-sided temporal stochastic process. 
\end{proof}

Hence, we further better the conclusion for the temporal stochastic process we arrived at in Section \ref{MT-seqrdv}.
%, which only concludes each run of our reanalysis output is just one realisation from a random variable or a trivial stochastic process, while collectively, these reanalysis outputs are just a collection of different realisations from trivial stochastic processes rather than one realisation from a temporal stochastic process, which would otherwise be more useful in terms of stochastic statistical modelling. 

\subsubsection{Spatial Stochastic Process} 
\label{pers2:sp}
To prove the reanalysis outputs of 4D-Var DA at a given time $t$ is one realisation from a spatial stochastic process, we will discuss two scenarios, one is that the grid locations are completely ordered like temporal indices in one dimension, and the other is that the grid locations are not ordered. 

For the first scenario, the proof idea will be similar to what we have seen in the proof of Theorem \ref{cams-output} and can be found in appendix \ref{app:sp_pers2}, and for the second scenario, we will need more additional concepts such as cylinder sets. For convenience, all of the following proofs are for reanalysis outputs at a given time $t$, hence we omit the time index $t$. 

We state the proof ideas for the non-ordered grid locations as below. 

Let $\{ s_{j'}: j' \in \mcalN (s_j), j, j' \in \mbbZ^{+} \}$ be a collection of location indices that are within the first-order neighbourhood of location $s_j$, e.g, for $j = 1$, we have $\{ s_{j'}: s_{j'} \in \mcalN (s_1) \}$, etc. We then collect them into $\mcalS = \{ \{ s_{\mcalN(s_1)}\}, \{ s_{\mcalN(s_2)}\}, \ldots \{ s_{\mcalN(s_N)}\} \}$. 

Follow the definition in \citet[p.~202]{athreya2006measure}, we define a collection of real-valued function $\mbbR^{\mcalS}$ as 
\begin{align*}
    \mbbR^{\mcalS} = \{ \omega \mid \omega: s \rightarrow \mbbR ^P \}, 
\end{align*}
here, $P$ is the number of aerosol compositions at location $s$ at time $t$. 
Further, we define a \textit{finite dimensional cylinder set (f.d.c.s.)} $\mcalC$, 
for $\mcalS_1 = \{s_1, s_2, \ldots, s_N \} \subset \mcalS$, 
$\mcalC \subset \mbbR^S$, and 
\begin{align*}
    \mcalC = \{\omega: \omega \in \mbbR^S \; \& \; (\omega(s_1), \ldots, \omega(s_N)) \in B \}, 
\end{align*}
where $B$ is a Borel set, $B \in \mcalB(\mbbR^{NP})$, $NP = n$ as we have seen in Section \ref{MT-seqrdv}. And by \citet[p.~203]{athreya2006measure}, we know a collection of such a  finite-dimensional cylinder set $\mcalC$ is a $\sigma$-algebra $\mcalR^S$. And for one f.d.c.s., $(\mbbR^S, \mcalR^S, P(\mcalC))$ is a probability space by Caratheodory extension theorem, where $P(\mcalC) = \mu_{s_1, \ldots, s_N}(B)$. 

Then we define a projection map $\pi_{(s_1, s_2, \ldots, s_N)} : \mbbR^S \rightarrow \mbbR^{NP}$, $\forall (s_1, s_2, \ldots, s_N) \in \mcalS^N$, $1 \leq N < \infty$, 
\begin{align*}
    \pi_{(s_1, s_2, \ldots, s_N)}(\omega) = (\omega(s_1), \omega(s_2), \ldots, \omega(s_N)), 
\end{align*}
in particular, for one location $s$, 
$\pi_s(\omega) = \omega(s) \triangleq \omega_s \equiv \bfy_s$, which is the observation at location $s$ at time $t$, hence the collection $\{ \omega_s:s \in \mcalS \} \subset \mcalA $ is the $\sigma$-algebra $\mcalA$ of $\Omega$, on which our measurable function $\mcalG$ is defined as seen in Section \ref{MT-seqrdv}.

We also define $\Delta s$ to be the Euclidean distance unit between any two locations $s_j, s_{j'}$ within the same first-order neighbourhood, i.e., $\Delta s = \parallel s_j - s_{j'} \parallel$, then the measurable and invertible transformation function $T: \Omega \rightarrow \Omega$ can shift $\omega$ at any one of the locations within the first-order neighbourhood of $s_j$ by $\Delta s$ unit to remain equivalent as $\omega_{s_j}$, i.e., 
$ T^{\Delta s} \omega_{[\mcalN(s_j)]_i} = \omega_{s_j}$, here $[\mcalN(s_j)]_i$ means any one of the locations in the first-order neighbourhood of $s_j$. 

So the reanalysis output at a certain location $s_j$ at a given time $t$ (omit) is 
\begin{align*}
    \bfx_{s_j}^A = \bfX_{s_j}^A (\omega_{s_j}) = \mcalG(T^{\Delta s} \omega_{[\mcalN(s_j)]_i}) = \bfX_{s_j}(\omega_{[\mcalN(s_j)]_i})
\end{align*}

For illustration, assume $\{ s_2, s_3, s_4, s_5 \}$ are within the first-order neighbourhood of $s_1$, this also implies $s_1$ is in the first-order neighbourhood of each of the locations in $\{ s_2, s_3, s_4, s_5 \}$, then we have 
\begin{align*}
    \bfx_{s_2}^A = \bfX_{s_2}^A (\omega_{s_2}) = \mcalG (\omega_{s_2}) 
    = \mcalG (T^{\Delta s} (\omega_{[\mcalN(s_2)]_i})) = \mcalG T^{\Delta s} ((\omega_{[\mcalN(s_2)]_i})) = \bfX_{s_2}(\omega_{s_1})  \\
    %\bfx_{s_3}^A = \bfX_{s_3}^A (\omega_{s_3}) = \mcalG (\omega_{s_3}) 
    %= \mcalG (T^{\Delta s} (\omega_{[\mcalN(s_3)]_i})) = \mcalG T^{\Delta s} ((\omega_{[\mcalN(s_3)]_i})) = \bfX_{s_3}(\omega_{s_1})  \\
    \vdots \qquad \qquad \qquad \qquad \vdots \qquad \qquad \qquad \qquad \vdots \qquad \qquad \qquad \qquad \vdots \qquad \qquad \qquad \qquad\\
    \bfx_{s_5}^A = \bfX_{s_5}^A (\omega_{s_5}) = \mcalG (\omega_{s_5}) 
    = \mcalG (T^{\Delta s} (\omega_{[\mcalN(s_5)]_i})) = \mcalG T^{\Delta s} ((\omega_{[\mcalN(s_5)]_i})) = \bfX_{s_5}(\omega_{s_1})
\end{align*}

From the leftmost side of the above equations, we have a collection of our reanalysis outputs $\{ \bfx_{s_2}^A, \bfx_{s_3}^A, \ldots \bfx_{s_5}^A\}$ 
at a given time across all locations within the first-order neighbourhood of $s_1$, 
and on the rightmost sides, we observe one realisation $\{ \bfX_{s_2}(\omega_{s_1}), \bfX_{s_3}(\omega_{s_1}), \ldots, \bfX_{s_5}(\omega_{s_1})\}$
of a spatially evolved random variable for a fixed sample element $\omega_{s_1}$. Hence, the collection of the 4D-Var DA reanalysis outputs at a given time $t$ is one realisation from a spatial stochastic process.

To complete our investigation into the stochasticity of the 4D-Var reanalysis outputs data, we inspect different errors associated with each run of the reanalysis output and their corresponding properties in the following section.

\section{The Errors and Their Properties}
\label{errors}

In Section \ref{mechanism}, we analysed detailed procedures of how each run of the reanalysis output is generated. We now base on these to further dissect the mechanism of how the corresponding errors are associated, categorized, propagated and correlated.

\subsection{Temporal-wise Dissection of Errors }
\label{temp-error}
In the first run of DA when $t = 0$, $\ttilde = 0 \sim \Delta t$, background state $\bfx^B_0$, initial analysis state $\bfx_0$ and observations $\bfy_{\ttilde}$ all have their own sources of errors, for example, both background state $\bfx^B_0$ and initial analysis state $\bfx_0$ are set to equal a guess which has its intrinsic random error, and due to different measurement instruments and instrument types etc., observations $\bfy_{\ttilde}$ collected during $\ttilde = 0 \sim \Delta t$ naturally contain random observational measurement error. 

Meanwhile, modelled observations $\bfH_{\ttilde}^o \bfx_{\ttilde}$ inherits errors from $\bfx_{\ttilde}$ via an NWP model and $\bfx_0$, i.e., $\bfx_{\ttilde} = \bfM_{\ttilde} \cdot \bfM_{\ttilde - \delta t} \cdots \bfM_{0+\delta t} \cdot \bfx_0$, in which the model input $\bfx_0$ has intrinsic random error and the model itself has systematic discrepancies which are due to spatial proximity (from $\bfH^o$) and temporal proximity (from $\bfM \ldots  \bfM$), %see \citet[Section.~3]{rougier2007probabilistic} 
or limited spatial and temporal resolutions simulated by the NWP model. 

And since each of these error-associated values $\bfx^B_0$, $\bfx_0$, $\bfy_{\ttilde}$ and $\bfH_{\ttilde}^o \bfM_{\ttilde} \cdot \bfM_{\ttilde - \delta t} \cdots \bfM_{0+\delta t} \cdot \bfx_0$ all plays a role in the generation of 
the first reanalysis output $\bfx^A_0$, 
these associated errors (either random or systematic) therefore become an inseparable part of the first-run reanalysis output $\bfx^A_0$.

In the second run of DA when $t = \Delta t $, $\ttilde = \Delta t \sim 2 \Delta t$, the background state $\bfx^B_{\Delta t}$ is obtained by applying an NWP model onto the first-run reanalysis output $\bfx^A_0$, that is $\bfx^B_{\Delta t} = M_{\Delta t} \cdot M_{\Delta t - \delta t} \cdots M_{0+\delta t} \cdot \bfx^A_0$. So the errors contained in the first-run reanalysis output $\bfx^A_0$ are now propagated into the background state of the second run of DA, i.e., $\bfx^B_{\Delta t}$. Similar to the analysis above for the first run, this error propagation procedure involves both the discrepancies in the NWP model and errors in the model input $\bfx^A_0$, which is received from the last run of DA.
The newly collected observations $\bfy_{\ttilde}$ surely have their own random measurement errors. And all of these errors again become a part of the reanalysis output $\bfx^A_{\Delta t}$ of the second run from which it cannot devoid. 

The same analytical logic applies to all the following runs of DA. 
We summarize and visualize these error dissections and propagation in Figure \ref{appfig:error_propagate}. %in Appendix \ref{appendix:Fig}.
\begin{figure}[!ht] 
    \centering
    \includegraphics[width=1\textwidth]{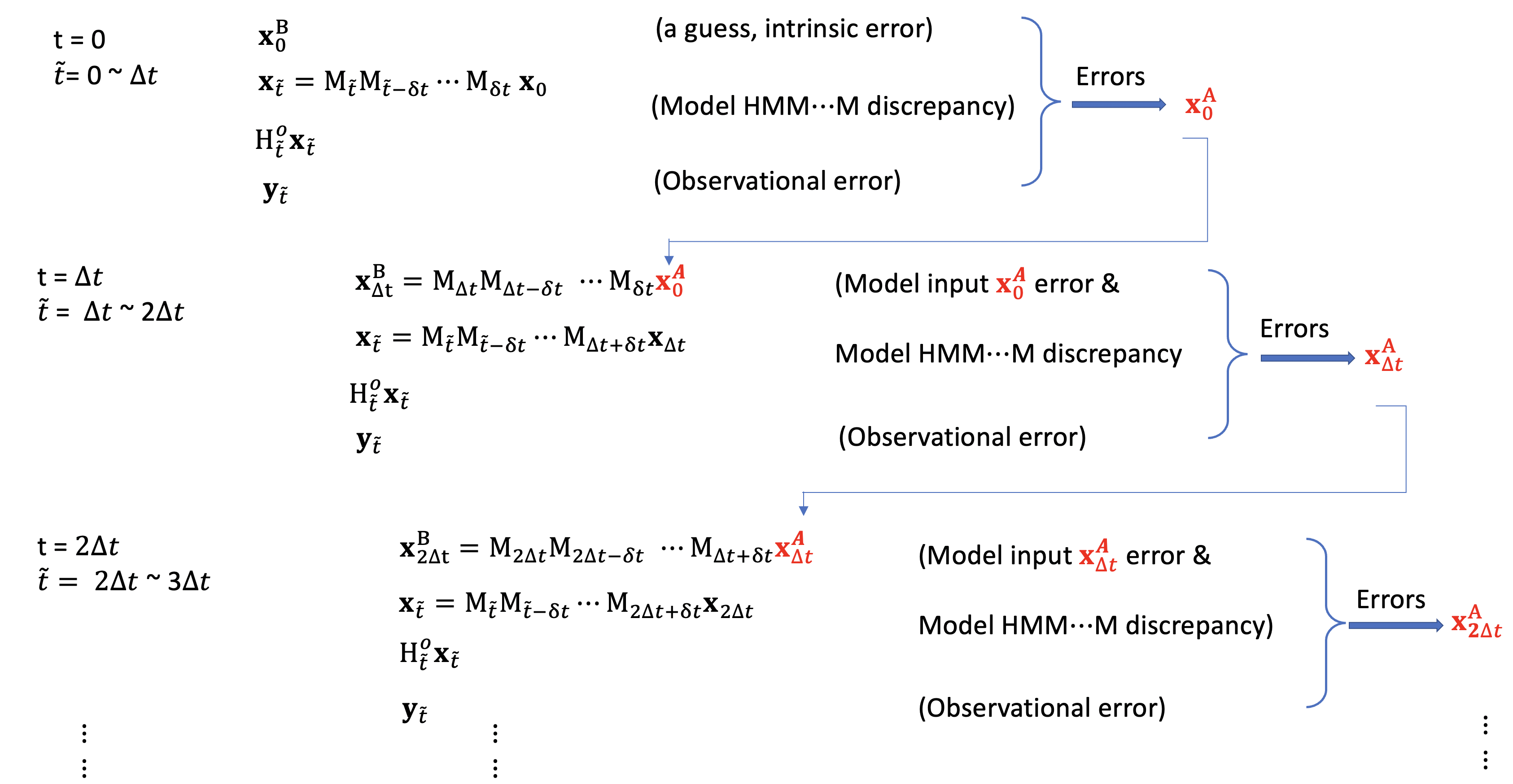}
    \caption{Schematic representation of dissections and propagation of errors associated with reanalysis output from each run of DA.}
    \label{appfig:error_propagate} 
\end{figure}

So, the reanalysis output of each run of DA or each of the temporal-wise reanalysis outputs $\{ \bfx_0^A, \bfx_{\Delta t}^A,\bfx_{2\Delta t}^A, \ldots \}$ contains extensive errors, mainly including NWP model input errors, NWP model discrepancies, and observational measurement errors. The random errors in the initial input $\bfx_0$ are to some extent reducible by calibration with some monitoring observations, see \citet[Section.~4]{rougier2007probabilistic} and \citet[Section.~3]{buizzachaos}, rather than using pure guess, while the model discrepancies and observational measurement errors are irreducible. 

And although the errors in each run of reanalysis output are propagated into the next runs of model input, therefore are correlated across different DA runs or correlated temporal-wisely, the sources of these errors are however different and are prone to be mutually independent. Specifically, for the first initial run, NWP model input $\bfx_0$ errors are from modellers' guess; model discrepancies are mainly due to the systematic proximity in space and time as well as limited spatial-temporal resolution of model simulation; and observational measurement errors are mainly from instruments etc. So for the first initial run of reanalysis output, the NWP model input error, NWP model discrepancy, and observational measurement error are reasonably assumed to be mutually independent, or put another way, they have no interactions. Mathematically, for the first run of reanalysis output, 
\begin{align*}
    \mbox{NWP model input error} \indep \mbox{NWP model discrepancy} \indep \mbox{observational measurement error}, 
\end{align*}
where ``$\indep$'' here means ``independent''.

However, if we keep thinking about each of the following runs of DA, we will arrive at a different situation. 
We know 
errors in the model input of this run are \textit{de facto} the errors in the reanalysis output from the last run, and it can be further decomposed into model-input error, model discrepancy and observational measurement error of the last run. And we also know that model discrepancy remains unchanged across different DA runs once the explicit NWP model formula has been specified (see discussion below equation \eqref{eq:bfx_tA_compact} in appendix). These together imply errors in the model input of this run are correlated with the model discrepancy of this run which is the same as that of the last run. 
Further, since observational measurement error and model input change each run, errors in the model input of this run can still be postulated to be independent of the observational measurement error of this run; meanwhile, the model discrepancy of this run and observational measurement error of this run still remain independent as well due to their independent origins. Write mathematically, for each of the following runs of DA except the first initial run,
\begin{align*}
    \mbox{NWP model input error} &\nindep \mbox{NWP model discrepancy} \\
    \mbox{NWP model input error} &\indep \mbox{observational measurement error}\\
    \mbox{NWP model discrepancy} &\indep \mbox{observational measurement error}
\end{align*}

\subsection{Spatial-wise Dissection of Errors}
To analyze spatial-wise errors, we choose a certain time $t$. 

At this time, the reanalysis output is 
$\bfx_t^A$ which can be further expanded within a finite spatial domain as 
$ \bfx_t^A \equiv
\{ \bfx^A_{s_1}(t), \bfx^A_{s_2}(t), \ldots\ \bfx^A_{s_N}(t) \}$.
%\equiv \{ \bfX^A _{s_1} (t = 0, \bfx^B_0, \bfy_{\ttilde}, H_t^o \bfx_t), 
%\bfX^A _{s_2} (t = 0, \bfx^B_0, \bfy_{\ttilde}, H_t^o \bfx_t), %\ldots \bfX^A _{s_N} (t = 0, %\bfx^B_0, \bfy_{\ttilde}, H_t^o \bfx_t) \}$.
The errors that $\bfx_t^A$ inherited from $\bfx^B_t$, $\bfx_t$, $\bfy_{\ttilde}$ and $\bfH_{\ttilde}^o \cdot \bfM_{\ttilde} \cdot \bfM_{\ttilde - \delta t} \cdots \bfM_{t+\delta t} \cdot \bfx_t$ as analysed above are now down onto each of the reanalysis values at each grid location. From the above Section \ref{temp-error}, we know at each grid location, the errors associated with the reanalysis outputs can be mainly characterized into NWP model input error, NWP model discrepancy, and observational measurement error, and these three types of errors have their own specific dependence/independence relationship due to their corresponding error sources.

Further, if we view these three types of errors contained in one run of reanalysis output altogether as a whole, then the correlations among these errors at 
different grid locations have two special cases: 
\begin{enumerate}
    \item reanalysis outputs at different grids all have the same amount of errors, so the correlations among reanalysis-output errors across different grids all equal one;
    \item reanalysis outputs at different grids all have their own different individual errors, so the correlations among reanalysis-output errors across different grids all equal zero, which is equivalent to grid-wise independent errors. 
\end{enumerate}
The general case is that the correlations of the errors (associated with a given run of reanalysis output) across different grid locations are between 0 and 1.
And at each grid, the errors associated with the reanalysis output at this specific grid and this specific time can still be categorized into different types according to different sources as analysed in Section \ref{temp-error}.

\subsection{Composition-wise Dissection of Errors}
\label{cmpt-error}
And as mentioned in Section \ref{mechanism}, at a given time $t$ and a given grid location $s$, $\bfx^A_{s} (t) = \bfX^A_{s} (\omega_{\bfy_{\ttilde}}; t)$ collects all the values for different atmospheric compositions, i.e., 
$
\bfx^A _{s} (t) = \bfX^A _{s} (\omega_{\bfy_{\ttilde}};t) 
= [X^{A_{PM25}} _{s}(\omega_{\bfy_{\ttilde}};t), X^{A_{BC}} _{s}(\omega_{\bfy_{\ttilde}};t), X^{A_{SU}} _{s}(\omega_{\bfy_{\ttilde}};t), \ldots] 
= [x^{A_{PM25}} _{s}(t), x^{A_{BC}} _{s}(t), x^{A_{SU}} _{s}(t), \ldots]
$
where each of these scalars possesses both spatial features and temporal ones.

So the errors associated with reanalysis output $\bfx^A _{s} (t)$ at a given time $t$ and a grid $s$ are passed onto each of the compositions' values at this specific time and grid location. Two extreme cases are either each atmospheric composition at this grid location at this time has the same amount of errors, that is the correlations among errors across different atmospheric compositions at this location at this time are all equal to one, or the correlations are all equal to zero meaning the reanalysis-output errors across different atmospheric compositions at this location and this time are completely different. 
And in general, the correlations of errors across different atmospheric compositions at a given location and a given time are between 0 and 1. 

And same as above, the errors associated with each atmospheric composition at a given grid and a given time can be categorized into different types corresponding to different error sources, and different types of errors have their mutual dependence/independence relations.

\subsection{The Role of Error Covariance Matrices in $J$}
Note that the role of two error covariance matrices $\bfB$ and $\bfR$ in the cost function $J$ mentioned in Section \ref{mechanism} is just to re-scale the errors in the background state $\bfx^B$ and observation $\bfy$ so as
to standardize them for general comparisons, but multiplying the inversion of these two matrices (i.e. $\bfB^{-1}$ and $\bfR^{-1}$) are unable to eliminate the random errors in $\bfx^B$ and $\bfy$.

%\subsection{Section Summary}
%So to summarize, reanalysis outputs of 4D-Var DA (i.e., the ECMWF CAMS reanalysis data) have extensive random errors, and such random errors as a whole are correlated temporally, spatially at a given time and across different aerosol compositions at a given time and a given grid location.

%For reanalysis outputs across all DA runs (i.e. temporal-wisely), or each reanalysis output at a given run or time (i.e. spatial-wisely), or reanalysis output of each run at a give grid location (i.e. component-wisely), the associated random errors can be further categorized into different types according to different errors sources. These different types of random errors have different mutual dependence/independence relationships. 

\section{Discussion} 
\label{discussion}
%-remind your readers figs tables in the text to help them understand your presentation. 
%-present tense for any general conclusions, and past tense about your own results.
%-active 

This research is motivated by the unclear stoachstic property of the ECMWF CAMS reanalysis data set which contains resourceful information for researchers in many environmental-related fields, e.g., spatio-temporal modelling, public health, climate change, environmental intelligence etc. 
And the aim of this paper is to investigate the stochastic property of these reanalysis outputs generated from computers via a 4D-Var DA mechanism.

%What's the methods? past tense 
%your methodolgy & corresponding results
%    - why you do so
%    - why you 
%    - present minor conclusions as you go along

%What do these results mean? why your research is important to research community? 
%- relation your results to your field of research
%- enable spatio-temporal stochastic modelling framework
%- justify additive error term and no interaction.

%What's the impact of these. 
%- a parallel framework to black box, open black box

We first gave a clear exposition of the 4D-Var DA mechanism on which later proofs and reasoning stand. 

We adopted %stochastic-process-constructing 
measure theory and proved the existence of stochasticity in these reanalysis outputs from two perspectives. 
Specifically, from the perspective of a sequence of random variables, we demonstrated the tangible existence of temporal and spatial stochastic processes associated with the reanalysis outputs of the 4D-Var DA %(i.e., ECMWF CAMS reanalysis data)
, and these processes are essentially digitised versions of real-world hidden temporal and spatial processes,
see Lemma \ref{lemma 4.1}, \ref{lemma 4.2} and \ref{lemma 4.3}; and in particular, we confirmed at a given time $t$, the corresponding reanalysis outputs are one realisation from a spatial stochastic process, see Theorem \ref{sp-stochastic_pro};
from the perspective of an abstract dynamic system, we further revealed that the collection of the reanalysis outputs from all runs of 4D-Var DA is one realisation of a temporal stochastic process, see Theorem \ref{cams-output}. The existence of spatial process and one realisation from the spatial process are also proved from this perspective, see Section \ref{pers2:sp}.

These results mean, in practice, we can treat the ECMWF CAMS reanalysis data set the same as those observational measurements obtained from monitoring equipment, which are usually deemed to have intrinsic randomness,   
and therefore stochastic spatio-temporal models are applicable to this ECMWF CAMS reanalysis data set. 

We also comprehensively analysed different sources for different types of errors associated with these reanalysis outputs, see Section \ref{errors}. In general, they are model input error which is reducible; model discrepancy and observational measurement error, which are irreducible. In addition, we deciphered the mutual dependence/independence relationships among these three types of errors, which altogether serve as definite guidance on the modelling of error terms. 
Specifically, 
the mutual independence between model discrepancy and observational measurement error allows us to model the reanalysis outputs data using additive error terms and do not need to model any interactions between them.

%such as spatial/temporal random error terms (or \textit{random effects}) stemmed from the model discrepancy, and measurement error term which is usually Gaussian, and moreover, we . 

Altogether, with the proven existence of spatial and temporal stochastic processes in the reanalysis data set as well as the mutual independence between model discrepancy and measurement error, 
those standard spatio-temporal stochastic modelling frameworks, 
for example, \textit{Data = Covariates + Spatial random effects + Temporal random effects + (Spatio-temporal interaction) + Random measurement error} \citep[p. 304-305]{cressie2015statistics} is applicable for the ECMWF CAMS reanalysis data set. 

We, therefore, expand the utility scope of the reanalysis outputs of 4D-Var DA (i.e., ECMWF CAMS reanalysis data set)
beyond those empirical utilities such as climatological computing (mean, percentiles, etc.), trends studying, geographical visualisation mapping, etc. and limited statistical applications such as uncertainty quantification and data fusion. Instead, these reanalysis data can be modelled by stochastic statistical models especially stochastic spatio-temporal models solely without fusion or ensemble.

From the error analysis in Section \ref{errors}, we realise that the spatial and temporal information was involved and addressed by two matrices $\bfH$ and $\bfM$ in the generation process of reanalysis data $\bfx^A_t$, which is non-stochastic. And from our proofs, we know the reanalysis outputs $\bfx^A_t$ are essentially realisations from digitised versions of real-world hidden spatial and temporal processes, so these two types of information can actually be better addressed by using a stochastic modelling scheme. 

Meanwhile, in Section \ref{cmpt-error}, we also realise that the errors are generally correlated across different aerosol components at a given time and grid, yet in Section \ref{mechanism}, we mentioned that the 4D-Var DA were implemented univariately for convenience in practice, hence a multivaraite stochastic spatio-temporal modelling scheme may be able to better the reanalysis output data further, 
and the refined result can then, in turn, benefit wider environmental impact studies such as public health, climate change, environmental intelligence etc. 

The conclusions of this paper alone also serve as a cogent theoretical foundation for spatio-temporal modellers and environmental AI researchers to embark on their research directly should they intend to use this ECMWF CAMS reanalysis data set with stochastic models.

One thing to mention is although this paper emphasized the 4D-Var DA mechanism, the real method used to produce the ECMWF CAMS reanalysis data set is more precisely the \textit{incremental} 4D-Var. 
However, no matter whether it's 
the key features, function structures or properties of this incremental 4D-Var are all exactly the same as those of 4D-Var.
The incremental 4D-Var just makes some minor modifications towards the state components of the cost function $J$ (i.e. $\bfx^B_t$, $\bfx_t$ and $H_{\ttilde}^o \bfx_{\ttilde}$)
by separating each of them into a self-defined reference state $\bfx_t^R$ and an incremental state $\delta \bfx_t$, and replacing the minimisation target from $\bfx_t$ to $\delta \bfx_t$, meanwhile setting the NWP model $\bfM$ to a linear form instead of a nonlinear one as in 4D-Var. 
So, the main purpose of modifying the 4D-Var into an incremental one is to lessen the computational burden. Since the original 4D-Var has a more understandable meaning for its various components and structures, see Section \ref{mechanism}, and therefore has a clearer demonstration effect, we emphasized mainly the 4D-Var. For full details of incremental 4D-Var, see \citet[p.~17-19]{Ross4DVARLect}. 

In addition, as pointed out by \citet[Section 2.4]{wikle2007bayesian}, 
optimization problems can be expressed equivalently as variational problems, and for high-dimensional tasks, variational method is more computationally efficient. Therefore, in real practice, 4D-Var DA reanalysis adopts variational inference method to obtain the outputs 
\citep[Section 3.3]{bannister2001elementary}.

Another thing to note is although the proofs in this paper focus mainly on the existence of discrete stochastic processes both temporally and spatially from a practical assimilation operational perspective, these proofs can be extended to continuous situations without any effort by just setting the temporal increment $\Delta t$ to infinitely small to achieve a continuous temporal process; and for continuous spatial processes, either by increasing the spatial domain through \textit{increasing-domain asymptotics} \citep[p.~350]{cressie1993statistics} 
or by setting the number of grid locations $N$ within a finite domain to be larger and larger via \textit{infill asymptotics} \citep[p.~350]{cressie1993statistics}. The corresponding changes in terms of the symbol would just be replacing the subscripts, e.g., $X_t$, $X_{s_i}$ by arguments in function brackets, i.e., $X(t)$, $X(s_i)$. 

Regarding future work, there are actually other ideas to prove the existence of stochasticity of the reanalysis data from the 4D-Var DA.

Notice in Section \ref{exam-stoch}, we treated the background state $\bfx_t^B$ as a constant due to the fact that it is obtained from the last-run reanalysis output whose randomness is yet to be proven and we focused only on the random observations $\bfy_{\ttilde}$. We could otherwise ignore the assimilated observations $\bfy_{\ttilde}$ at each run, and only focus on the relationship between reanalysis output $\bfx_t^A$ at time $t$ and the background state $\bfx_t^B$ at time $t$, which is essentially the reanalysis output at time $t-\Delta t$, and this is then a classical deterministic dynamic system. However, even from this deterministic system, we are still able to connect with stochasticity as long as the initial state of this deterministic system is random, see, e.g., \cite{berliner1992statistics}. %, \cite{lorenz1963deterministic}. 
The random initial state is usually justified by our imperfect knowledge to initialize the system, and usually has a tremendous impact on the ability of a deterministic system to decide its future value definitely and uniquely, hence, usually leads to a status called \textit{deterministic chaos} \citep[p.~1]{chan2001chaos}, in which one is typically unable to differentiate between \textit{chaotic randomness} and stochastic-process randomness. This is also discussed in \cite[p.~58-59]{cressie2015statistics}.

\section*{Acknowledgements}
This work is supported by The Alan Turing Institute through a Turing Doctoral Scholarship.
The first author is grateful to Prof.~Peter Ashwin who provided suggestions on the proofs. The first author's gratitude also extends to Dr.~Antje Inness who provided the background information about the time evolution matrix in real practice.
%is in great debt to Dr.~Ross Bannister who explained the mechanism of the 4D-Var in great patience and offered lots of inspiring discussions as well as encouragements. 
%The first author's gratitude also extends to 

\clearpage
\printbibliography[title={References}]
\clearpage

\begin{appendices}
\section{Time Evolution of Model State by NWP}
\label{app:NWP}
The second quadratic term, e.g., $\sum_{\ttilde = t}^{t + \Delta t} (\bfy_{\ttilde} - H_{\ttilde}^o \bfx_{\ttilde})^T(\bfy_{\ttilde} - H_{\ttilde}^o \bfx_{\ttilde})$ in the cost function J \eqref{J_fun} involves not only the model state $\bfx_{t}$ at time $t$, but also model states $\bfx_{t + \delta t}$ at discrete incremental time steps $t + \delta t$ within the time window, e.g., $t \sim (t + \Delta t)$. Such model states  $\bfx_{t + \delta t}$ are acquired by evolving from earlier state $\bfx_t$ through a functional form $m(.)$ driven by physical law (hence is usually non-linear), that is, for any time-step $t$, 
\begin{align*}
    \bfx_{t + \delta t} = m(\bfx_t).
\end{align*}
The above equation indicates each component in the vector $\bfx_{t + \delta t}$ is a non-linear combination of each component of vector $\bfx_t$. 

The non-linear $m(\bfx_t)$ is then approximated with linear terms using Taylor expansion expanded at a certain point, e.g., $\bfx_t^B$ or simply $\bold{0}$. 

For example, by Taylor expansion of $\bfx_{t + \delta t} = m(\bfx_t)$ at $\bfx_t^B$, we get
\begin{align*}
    \bfx_{t + \delta t} &= m(\bfx_t) \\
    &= m(\bfx_t^B) + m'(\bfx_t^B) (\bfx_t - \bfx_t^B) + o(\cdot), 
\end{align*}
here, $m'(\bfx_t^B)$ is the first derivative of $m(\bfx_t)$ with respect to (w.r.t.) each component of vector $\bfx_t$ and then evaluated at $\bfx_t = \bfx_t^B$. As a demonstration, we assume $\bfx_t$ is a 2-D vector, then 
\begin{align*}
    m'(\bfx_t^B)  = \left. 
    \left[ \begin{array}{cc}
    \frac{\partial [m(\bfx_t)]_1}{\partial[\bfx_t]_1} & \frac{\partial [m(\bfx_t)]_1}{\partial[\bfx_t]_2} \\
    \frac{\partial [m(\bfx_t)]_2}{\partial[\bfx_t]_1} & \frac{\partial [m(\bfx_t)]_2}{\partial[\bfx_t]_2}  \end{array} \right] 
    \right|_{\bfx_t = \bfx_t^B} \triangleq M
\end{align*}

So, 
\begin{align*}
    \bfx_{t + \delta t} &= m(\bfx_t) \\
    &= m(\bfx_t^B) + M (\bfx_t - \bfx_t^B) + o(\cdot).
\end{align*}
And if the Taylor expansion is expanded at $\bold{0}$, then 
\begin{align*}
    \bfx_{t + \delta t} &= m(\bfx_t) \\
    &= m(\bold{0}) + M (\bfx_t) + o(\cdot)\\
    &\approx M \bfx_t, 
\end{align*}
where matrix $M$ contains all the known first derivative of $m(\bfx_t)$ evaluated at $\bold{0}$. 

Denote the $M$ in $\bfx_{t + \delta t} = M \bfx_t$ more specifically as $M_{t + \delta t}$ and follow the same idea of approximating one model state with Taylor expansion at $\bold{0}$, we get
\begin{align*}
    \bfx_{t + \delta t} &= M_{t + \delta t} \bfx_t \\
    &= M_{t + \delta t} M_t \bfx_{t - \delta t} \qquad \mbox{(by $\bfx_t = M_t \bfx_{t - \delta t)}$} \\
    &= M_{t + \delta t} M_t M_{t - \delta t} \bfx_{t - 2 \delta t} \qquad \mbox{(by $\bfx_{t - \delta t} = M_{t - \delta t} \bfx_{t - 2 \delta t)}$} \\
    \ldots \\
    &= M_{t + \delta t} M_t M_{t - \delta t} \ldots M_{2 \delta t} \bfx_{\delta t} \\
    &= M_{t + \delta t} M_t M_{t - \delta t} \ldots M_{2 \delta t} M_{\delta t} \bfx_0 \qquad \mbox{(by $\bfx_{\delta t} = M_{\delta t} \bfx_0)$}
\end{align*}
Hence, the description of the model state $\bfx_{\ttilde}$ evolution as $\bfx_t = M_{t} M_{t-\delta t} \ldots M_{\delta t} \bfx_0$. 

Note in real practice, $m(\cdot)$ is usually simplified to linear forms. 
\clearpage

\section{Structure of $\bfH_{\ttilde} ^o$}
\label{app:H}
Let $\bfs_i = (x_i, y_i)$ be the centroid of each grid. 
Observations at each of these centroids are obtainable, e.g., at centroids $(x_2, y_i)$, $(x_3, y_i)$, we have observations $O_{2i}$, $O_{3i}$. 
We want to know the value of an observation $O_{Ti}$ which is not at centroid but at a location $(x_T, y_i)$ in between $(x_2, y_i)$, $(x_3, y_i)$, that is $x_2 < x_T < x_3$.
This would require interpolation coefficients to smooth the two observations $O_{2i}$, $O_{3i}$, i.e., $O_{Ti} =  \alpha O_{2i} + \beta O_{3i}$. 
Here $\alpha = \frac{x_3 - x_T}{x_3 - x_2}$, and $\beta = \frac{x_T - x_2}{x_3 - x_2}$. 
And these interpolation coefficients are the elements in the $\bfH_{\ttilde} ^o$. 

We can see that each element in $\bfH_{\ttilde} ^o$ is within $(0, 1)$ if the locations are assumed to be ordered.

\section{Illustrative Derivation of Cost Function $J$ and $\nabla_{\bfx}J$}
\label{app:J&DJ}
For clear inspection, we omit the background-state error covariance matrix $\bfB$ and observational error covariance matrix $\bfR$ which are just scalars, and assume $\bfx_{\cdot}$and $\bfy$ are just 1-dimensional rather than n-dimensional vectors for now so as to see the quadratic structure clearer. 

In the first run, $t = 0$, $\ttilde = 0 \sim \Delta t$,
\begin{align*}
    J[x_0] &\propto \frac{1}{2} (x^B_0 - x_0)^2 + 
    \frac{1}{2} \{(y_0 - H_0^o x_0)^2 + 
(y_{\delta t} - H_{\delta t}^o x_{\delta t})^2 + \ldots + 
(y_{\Delta t} - H_{\Delta t}^o x_{\Delta t})^2 \} \\
&= \frac{1}{2} (x^B_0 - x_0)^2 + \frac{1}{2} \{(y_0 - H_0^o x_0)^2 + (y_{\delta t} - H_{\delta t}^o M_{\delta t}x_{0})^2
+ \ldots + (y_{\Delta t} - H_{\Delta t}^o M_{\Delta t} M_{\Delta t - \delta t}\ldots M_{\delta t}x_{0})^2 \}
\end{align*}

Take first derivative with respect to (w.r.t.) $x_0$
\begin{align*}
    \nabla_{x_0} J &\propto -(x^B_0 - x_0) 
    - H_0^o(y_0 - H_0^o x_0)
    %(-H_{\delta t}^0 M_{\delta t})(y_{\delta t} - H_{\delta t} M_{\delta t} x_{0}) 
    - \ldots + 
    (- H_{\Delta t}^0 M_{\Delta t} M_{\Delta t - \delta t} \ldots M_{\delta t} )(y_{\Delta t} - H_{\Delta t}^0 M_{\Delta t} M_{\Delta t - \delta t} \ldots M_{\delta t} x_0) \\
    &= -(x^B_0 - x_0) - \sum_{\ttilde=0}^{\Delta t} H_{\ttilde}^o M_{\ttilde} M_{\ttilde-\delta t} \ldots M_{\delta t +0}(y_{\ttilde} - H_{\ttilde}^o M_{\ttilde} M_{{\ttilde} - \delta t} \ldots M_{\delta t +0} x_0), \; M_0 = 1.
\end{align*}

In the second run, $t = \Delta t$, $\ttilde = \Delta t \sim 2 \Delta t$, 
\begin{align*}
    J[x_{\Delta t}] &\propto \frac{1}{2} (x_{\Delta t} ^B - x_{\Delta t})^2 + \frac{1}{2}
    \sum_{\ttilde=\Delta t} ^{2 \Delta t} (y_{\ttilde} - H_{\ttilde}^o x_{\ttilde})^2 \\
    &= \frac{1}{2} (x_{\Delta t} ^B - x_{\Delta t})^2 + \frac{1}{2} \{ 
    (y_{\Delta t} - H_{\Delta t}^o x_{\Delta t})^2 +  (y_{\Delta t + \delta t} - H_{\Delta t + \delta t}^o x_{\Delta t + \delta t})^2 + \ldots + (y_{2 \Delta t} - H_{2 \Delta t}^o x_{2 \Delta t})^2    
    \} \\
    &= \frac{1}{2} (x_{\Delta t} ^B - x_{\Delta t})^2 + \frac{1}{2} \{
    (y_{\Delta t} - H_{\Delta t}^o x_{\Delta t})^2 
     + \ldots + (y_{2 \Delta t} - H_{2 \Delta t}^o M_{2 \Delta t} M_{2 \Delta t-\delta t} \ldots M_{\Delta t + \delta t} x_{\Delta t})^2
    \}    
\end{align*}

Take first derivative w.r.t. $x_{\Delta t}$ for each term above and write in a compact summation form
\begin{align*}
    \nabla_{x_\Delta t} J \propto -(x_{\Delta t}^B - x_{\Delta t}) - 
    \sum_{\ttilde = \Delta t}^{2 \Delta t} 
    H_{\ttilde}^o M_{\ttilde} M_{\ttilde-\delta t} \ldots M_{\Delta t + \delta t} 
    (y_{\ttilde} - H_{\ttilde}^o M_{\ttilde} M_{\ttilde-\delta t} \ldots M_{\Delta t + \delta t} x_{\Delta t}), \; M_{\Delta t} = 1.   
\end{align*}

For general form of the first derivative of $J$ w.r.t. the desired initial state $\bfx_t$ at time t in n-dimensional vector, we have
\begin{align}
    \label{nablaJ}
    \nabla_{\bfx_t} J \propto -(\bfx_{t}^B - \bfx_t) - \sum_{\ttilde = t}^{t + \Delta t} (\bfH_{\ttilde}^o \bfM_{\ttilde} \bfM_{\ttilde-\delta t} \ldots \bfM_{t + \delta t})^T (\bfy_{\ttilde} - \bfH_{\ttilde}^o \bfM_{\ttilde} \bfM_{\ttilde-\delta t} \ldots \bfM_{t + \delta t} \bfx_t), 
\end{align}
in which $\bfM_{\ttilde} = I$ whenever $\ttilde = t$.

To get the desired initial state $\bfx_t$ at each time $t$, set the above equation \eqref{nablaJ} to $\textbf{0}$ and rearrange, 
\begin{align*}
    \bfx_{t}^B + \sum_{\ttilde = t}^{t + \Delta t} (\bfH_{\ttilde}^o \bfM_{\ttilde} \bfM_{\ttilde-\delta t} \ldots \bfM_{t + \delta t})^T \bfy_{\ttilde} = 
    (I + \sum_{\ttilde = t}^{t + \Delta t} (\bfH_{\ttilde}^o \bfM_{\ttilde} \bfM_{\ttilde-\delta t} \ldots \bfM_{t + \delta t})^T (\bfH_{\ttilde}^o \bfM_{\ttilde} \bfM_{\ttilde-\delta t} \ldots \bfM_{t + \delta t}) ) \bfx_t,
\end{align*}
where the coefficient matrix in the form of $(I + \sum_{\ttilde = t}^{t + \Delta t} (\bfH \bfM \dots \bfM)^T(\bfH \bfM \dots \bfM) )$ ahead of the desired initial state $\bfx_t$ must be non-singular hence invertible, so 
\begin{align}
    \label{eq:bfx_tA}
    \bfx_t \propto  (I + \sum (\cdot)^T (\cdot))^{-1} \bfx_t ^B + 
    \sum_{\ttilde = t}^{t + \Delta t} (I + \sum (\cdot)^T (\cdot))^{-1} (\bfH_{\ttilde}^o \bfM_{\ttilde} \bfM_{\ttilde-\delta t} \ldots \bfM_{t + \delta t})^T \bfy_{\ttilde}
\end{align}
where the ``$\cdot$" in $(\cdot)^T(\cdot)$ is $\bfH \bfM \dots \bfM$.  
And we could further denote the coefficient matrices ahead of $\bfx_t^B$ and $\bfy_{\ttilde}$ as $L$ and $K_{\ttilde}$ for convenience, and get
\begin{align}
    \label{eq:bfx_tA_compact}
    \bfx_t \propto L \bfx_t^B + \sum_{\ttilde = t}^{t + \Delta t} K_{\ttilde} \bfy_{\ttilde}
\end{align}

One thing to make clear is matrix $L = (I + \sum (\bfH \bfM \dots \bfM)^T(\bfH \bfM \dots \bfM ))^{-1}$ always remain the same form across different runs (or time-steps) since the elements in $\bfH$ only depend on relative (Euclidean) distance between locations, which don't change across different runs (or time-steps), and elements of $\bfM$ are the first derivatives of a non-linear form $\bfx_{t+\delta t} = m(\bfx_t)$ evaluated at $\textbf{0}$ for any $t$, so only depend on the relative difference of two time-steps, not on any explicit time $t$, hence won't change across different runs (or time-steps) as well. 

And for a given run, $\ttilde$ ranges from $t \sim (t+\Delta t)$, each coefficient matrix ahead of $y_{\ttilde}$ is in different form according to $\ttilde$, e.g., if $\ttilde = t$, then the coefficient matrix $K_{\ttilde}$ is in the form of $\bfH \bfM$, if $\ttilde = t + \delta t$,  the coefficient matrix $K_{\ttilde}$ is in the form of $\bfH \bfM \bfM$, etc. 
%\clearpage

\section{Proofs}
\subsection{Proof of Lemma \ref{lemma 4.3}}
\begin{proof}
\label{app:prooflemma4.3}
By Lemma \ref{lemma 4.1}, $\bfX_t^A$ is a measurable function of $\bfY_t$, hence at a given time $t$, each of the spatially indexed collection $\{ \bfX^A_{s_i}(t)$, $i = 1, 2, \ldots \}$ is a measurable function of $\bfY_t$, hence each of them is a spatially indexed random object defined on the same sample space $\Omega$ of $\bfY_t$. 
\end{proof}

\subsection{Proof of Theorem \ref{sp-stochastic_pro}}
\begin{proof}
\label{app:proofthrm4.1}
By Lemma \ref{lemma 4.3}, $\{ \bfX^A_{s_i}(t)$, $i = 1, 2, \ldots \}$  is a spatial stochastic process, and at given time $t$, $\omega_t \equiv \bfy_t$ is a fixed sample element, hence by definition, reanalysis outputs $\{ \bfx_{s_i}^A(t) = \bfX_{s_i}^A(\omega_t; t): i = 1,2, \ldots \}$ is one realisation or a sample path of the spatial stochastic process.    
\end{proof}

\subsection{Proof of Section \ref{pers2:sp}}
\label{app:sp_pers2}
In the first scenario, we assume the grid locations are ordered in one dimension, i.e., for two locations $s_i$ and $s_j$, $j > i$,  then we define measurable and invertible transformation function $T: \Omega \rightarrow \Omega$ as 
$T^{\Delta s} (\omega_{s_1}, \omega_{s_2}, \ldots )= (\omega_{s_2}, \omega_{s_3}, \ldots) $, where $\Delta s$ is the Euclidean distance unit between $s_i$ and $s_j$, $(j > i)$,
therefore, we have $\bfx_{s_1}^A = \bfX_{s_1}^A(\omega_{s_1})$,
\begin{align*}
    \bfx_{s_2}^A &= \bfX_{s_2}^A(\omega_{s_2}) = \mcalG (\omega_{s_2}) = \mcalG (T^{\Delta s} (\omega_{s_1})) = \mcalG T^{\Delta s} (\omega_{s_1}) = \bfX_{s_2}(\omega_{s_1}) \\   \bfx_{s_3}^A &= \bfX_{s_3}^A (\omega_{s_3}) =  \mcalG (\omega_{s_3}) = \mcalG (T^{\Delta s} (\omega_{s_2})) = \mcalG (T^{\Delta s} (T^{\Delta s} (\omega_{s_1}))) =
    \mcalG T^{2 \Delta t}(\omega_{s_1}) = \bfX_{s_3}(\omega_{s_1})\\
    \vdots &\qquad \qquad \vdots \qquad \qquad \vdots \qquad \qquad \vdots \qquad \qquad
\end{align*}

So, from the rightmost side of above equations, we have a collection of our reanalysis outputs at a given time, and from the leftmost side of above equations, we have one realisation from a spatially evolved random variable at a fixed sample element $\omega_{s_1}$.

%\counterwithin{figure}{section}
% ref1: https://tex.stackexchange.com/questions/192769/how-do-i-refer-appendix-in-latex-so-as-to-display-a-or-b-in-my-paper
% ref2: https://tex.stackexchange.com/questions/85776/change-figure-numbering-for-appendix

%\section{Visualisation of Figures}
%\label{appendix:Fig}
%\begin{figure}[htp]
 %   \centering
  %  \includegraphics[width=1\textwidth]{Figures/Publications_stochastic_4D_Var/4D-VAR_CYCLES_ttilde.png}
   % \caption{Schematic representation of obtaining reanalysis outputs via iterative data collections and 4D-Var assimilation procedures.}
    %\label{appfig:4Dvar_cycle}
%\end{figure}

\end{appendices}

\end{document}